\documentclass {amsart}

\usepackage{amsthm,amsmath,amssymb,stmaryrd,verbatim}
\usepackage[cmtip,arrow,matrix]{xy}

\newcommand{\mycommented}[1]{ }

        \newtheorem{thm}{Theorem}[section]
        \newtheorem{prop}[thm]{Proposition}
        \newtheorem{lemma}[thm]{Lemma} 
        \newtheorem{cor}[thm]{Corollary}

\theoremstyle{definition}
        \newtheorem{defn}[thm]{Definition}
       
        \newtheorem{examples}[thm]{Example}
        \newtheorem{construct}[thm]{Construction}
        \newtheorem{notation}[thm]{Notation}

\theoremstyle{remark}
         \newtheorem{remark}[thm]{Remark}

\numberwithin{equation}{section}

\newcommand{\cpt}{\text {Cpt}}
\newcommand{\ptcpt}{\text {Cpt}_*}

\newcommand{\ourcat}{\ensuremath{\mathbb A}_\nu}

\newcommand{\C}{\ensuremath{\mathbb C}}
\newcommand{\R}{\ensuremath{\mathbb R}}
\newcommand{\N}{\ensuremath{\mathbb N}}
\newcommand{\compacts}{\ensuremath{\mathcal K}}

\newcommand{\cstar}{\ensuremath{C}\text{*}}
\newcommand{\starhom}{\text{*-homomorphism}}
\newcommand{\starhoms}[1][ ]{\text{*-homomorphisms}{#1}}
\newcommand{\staralg}{\ensuremath{\mathbb T}}
\newcommand{\staralgs}[1][s]{\text{*-algebra{#1}}}
\newcommand{\substaralg}{\text{sub-*-algebra}}

\newcommand{\calg}{\cstar\text{-}alg}
\newcommand{\calgs}{\ensuremath{C}\staralgs}
\newcommand{\lmccstar}{\ensuremath{l.m.c.\text{-}\cstar}}
\newcommand{\LMCCSTAR}{\ensuremath{\mathbb L}}

\DeclareMathOperator{\Kof}{K}
\newcommand{\K}{\ensuremath{\Kof}}
\newcommand{\KK}{\K \K}

\newcommand{\cof}[1][?]{\ensuremath{\mathcal C(#1)}}

\newcommand{\bof}[1][?]{\ensuremath{\mathcal B(#1)}}
\newcommand{\bofh}{\bof[\mathcal H]}

\newcommand{\complete}[1]{{\overline{(#1)}^\nu}}

\DeclareMathOperator{\Hom}{Hom}
\DeclareMathOperator{\colim}{colim}

\DeclareMathOperator{\ho}{Ho}

\newcommand{\homeo}{\cong}
\newcommand{\union}{\bigcup}

\newcommand{\Smash}{\wedge}

\newcommand{\tensor}{\otimes}
\newcommand{\ouriso}{\thickapprox}

\newcommand{\Left}{\mathbb L}
\newcommand{\Right}{\mathbb R}
\newcommand{\lefta}[1][A]{\Left_{#1}}
\newcommand{\leftaatensorK}[1][A \atensorK \compacts]{\Left_{#1}}
\newcommand{\righta}[1][A]{\Right_{#1}}
\newcommand{\rightaatensorK}[1][A \atensorK \compacts]{\Right_{#1}}

\newcommand{\SEP}[1][ ]{seminorm extension property{#1}}
\newcommand{\CSEP}[1][ ]{cone \SEP[#1]}
\newcommand{\sCSEP}[1][ ]{stable \CSEP[#1]}

\newcommand{\proa}{\ourcat}
\newcommand{\ptcof}[1]{\ensuremath{\mathcal C_*\left(#1\right)}}
\newcommand{\atimes}{\varotimes}             
\newcommand{\ptatimes}{\varowedge}             
\newcommand{\atensorK}{\boxtimes}           

\newcommand{\pttop}{\ensuremath{\text{Top}_*}}

\begin{document}

\title[$KK$-groups as homotopy sets of a model category]{Realizing
  Kasparovs $KK$-theory groups as the homotopy classes of maps of a Quillen Model
  Category} 


\author[M. Joachim]
    {Michael Joachim}
\address{Mathematisches Institut \\
         Westf\"alische Wilhelms-Universit\"at M\"unster \\
         Einsteinstrasse 62 \\
         48149 M\"unster \\
         Germany}
\email{joachim@math.uni-muenster.de}

\author [M. W. Johnson]
        {Mark W. Johnson}
\address{Department of Mathematics\\
         Penn State Altoona\\
         Altoona, PA 16601-3760\\
         USA}
\email{mwj3@psu.edu}
\thanks{The authors gratefully acknowledge the support of SFB 478}


\date{}

\begin{abstract}
In this article we build a model category structure on the category
of sequentially complete \lmccstar-algebras such that the
corresponding homotopy classes of maps $Ho(A,B)$ 
for separable \calgs\ $A$ and $B$ coincide with the Kasparov groups $KK(A,B)$.
\end{abstract}

\maketitle

\section{Introduction}

The goal of the present article is to relate classical Kasparov $KK$-theory
with the concepts of Quillen's model category theory. 
The motivation for this project emerged from
one of the fundamental theorems of $KK$-theory due to Cuntz.

\begin{thm}[Cuntz]
        \label{htpyKK}
        Suppose $A$,$B$ are \calgs\ with $A$ separable and $B$
        $\sigma$-unital. Then there is a natural isomorphism
\[KK(A,B) \cong [qA,B \tensor \compacts],
\]
where $qA$ denotes the kernel of the fold map $A \ast A \to A$,
$\compacts$ stands for the C$\staralgs[{}]$ of compact operators
on a separable infinite dimensional Hilbert space, and
where $[?,?]$ stands for homotopy classes of $\starhoms$.
\end{thm}

The initial intuition was that replacing $A$ by $qA$ and $B$ by $B \tensor 
\compacts$
could be the cofibrant and fibrant replacement functors
for a corresponding model category structure on the category of \calgs\ which
somehow realizes the $KK$-groups as homotopy classes of maps.  
However, one of the first observations when trying to build a model category of
\calgs\ is that the category of \calgs\ is too small.
By the Gelfand-Naimark Theorem, commutative \calgs\ all correspond to
algebras of continuous functions on compact spaces.
However, looking for a model category structure on
the category of compact spaces or the category of pointed compact spaces
is not very reasonable since they do not contain arbitrary unions.
Instead one looks for
model category structures on the category of compactly generated
spaces, which can all be described as unions of their compact subspaces.
In the algebra setting this step would correspond to considering inverse limits
of \calgs\ instead of \calgs\ themselves. For technical reasons, we
consider a slightly bigger 
class of algebras, namely the so-called \lmccstar-algebras. By the
Arens-Michael 
decomposition theorem (\ref{Arens-Michael} below) any
\lmccstar-algebra can be embedded as a dense subalgebra 
into an inverse limit of \calgs, and an \lmccstar-algebra is complete
if and only 
if it is isomorphic to an inverse limit of \calgs. To make sure that the notion
of the tensor product in the category of \calgs\ agrees with the
tensor product in 
the category of algebras we are working with, we restrict to
$\nu$-sequentially complete \lmccstar-algebras for an infinite ordinal
$\nu$.

Our main theorem now is
\begin{thm}[$\equiv$ Theorem \ref{generalKK}]\label{Main_Theorem}
There is a model category structure on the category of $\nu$-sequentially
complete \lmccstar-algebras 
satisfying the following properties:
\begin{enumerate}
\item \label{realizesKK}
for a separable \calgs[ ]$A$ and a $\sigma$-unital \calgs[ ]$B$
there is a natural isomorphism
\[ Ho(A,B) \cong KK(A,B);\]
\item \label{weq_is_KKiso} a \starhom\ $A \to A'$ of separable \calgs\ is 
a weak equivalence for the model category structure if and only if it
is a $KK_*$-equivalence; 
\item \label{cofibrantly_generated}
the model category structure is cofibrantly generated;
\item \label{objects_are_fibrant}
each object is fibrant.
\end{enumerate}
\end{thm}
While properties \eqref{realizesKK} and \eqref{weq_is_KKiso} 
were the designing criteria, we wanted \eqref{cofibrantly_generated}
and \eqref{objects_are_fibrant} 
to hold as they are convenient from a model category theory point of view. 
Whenever a model category structure is cofibrantly generated,
there are so-called basic cells and an inductive process for building
cellular approximations. 
Moreover, one essentially needs only to understand the basic cells in
sufficient detail in order to build the whole model structure using
Quillen's Small Object Argument. By Property \eqref{objects_are_fibrant}
the model category at hand is right proper, which implies that homotopy
pullbacks (and general homotopy limits)
behave well in our category. We do not know whether or not our model
category is left proper (and hence proper altogether).

Since we wanted property \eqref{objects_are_fibrant} to hold,
of course it shouldn't be necessecary to apply any fibrant replacement functor
to an object $B$
for determining a homotopy set $Ho(A,B)$. For that reason, we do not
start building our model category from Cuntz's Theorem (\ref{htpyKK}) 
but a variant of it, based on the existence of an adjoint to tensoring 
with the compact operators.
For the category of inverse limits of \calgs, such an adjoint was considered
by Phillips, cf.~\cite[Proposition 5.8 on page 1084]{Phillips(1989)}.
For an inverse limit of \calgs\ $A=\lim_\alpha A_\alpha$,
let $W(A)$ denote the left adjoint of the functor which associates to an
inverse limit of \calgs\ $B=\lim_\beta B_\beta$ the
inverse limit $\lim_\beta B_\beta \tensor \compacts$.
We then have

\begin{thm}[Cuntz, Phillips]
For a separable C$\staralgs[]$ $A$ and a $\sigma$-unital C$\staralgs[]$
$B$ there is a natural isomorphism
\[ KK(A,B) \cong [W(qA), B].\]
\end{thm}

We now take this Theorem as our starting point and 
build our model category structure
with a weakened version of the standard lifting technique from the
Serre structure on 
$\pttop$, the category of pointed topological spaces. 
A variation is necessary because the usual adjoint pair is replaced
by a weaker condition,
where the left adjoint is only defined on compact spaces
rather than arbitrary spaces.  However, this technical detail causes
no serious added difficulty, but a detailed argument is provided in
Section \ref{lifting}.
The most difficult technical portion of our arguments arise from two
other facts. 
First, we cannot work completely within the category of compactly
generated spaces, 
as \lmccstar-algebras (other than \calgs) are typically not compactly
generated. 
For that reason, questions concerning the exponential law for mapping
spaces have to be treated 
with a lot of care, as detailed in Section \ref{mapping}.
The second technical difficulty is that we need to work with
colimits in the category of \lmccstar-algebras and colimits in that category
typically are not very easy to deal with. For that reason we introduced
what we call the \SEP[]. By definition a \starhom\ has the \SEP[]\ if and only if
each continuous \cstar-seminorm on the domain is the pullback of some \cstar-seminorm
of the target.  This condition is analogous
to the notion of topological embedding which plays a crucial role
when applying Quillen's Small
Object Argument in building the Serre structure on $\pttop$.   
Sequential colimits over \starhoms having the \SEP[]\ behave much better
than in the general case, and we then have to show that various
relevant \starhoms\ in fact do have this property in Section \ref{studysep}.

On a practical level, our main theorem (and Section \ref{both}
generally) provides a solution to an open 
question posed by Hovey, Problem 8.4 in his foundational book
\cite{Hovey}.  The model category structure also provides a variety of
new techniques which might be useful in studying $KK$-theory, along
with a common language for discussing $KK$-theory with homotopy
theorists familiar with those techniques.  As we 
mentioned above, there will be well-behaved notions of homotopy inverse
limits, which includes homotopy fibers, in our structure.  A variety
of Spectral Sequence techniques can now be applied to computing the
extended notion of $KK$-theory provided by the homotopy classes of maps in our
model structure (cf.~Remark \ref{extendingKK}).  
If one wanted to build a better model for some
specific purpose, particularly with more familiar basic cells,
the well-understood notion of a Quillen
equivalence now makes that relatively straightforward.  
If one wanted
to study how formally inverting a certain map might affect
$KK$-theory, one could now apply the theory of localizing model categories,
as introduced e.g.~in \cite{Hirschhorn(2003)}.  

Now a word on the organization of the article.
We begin with several technical results on the category
of topological \staralgs\ in Section \ref{technical}.  Section
\ref{mapping} is a 
derivation from \cite{Steenrod(1967)} and \cite{tomDieck-Kamps-Puppe(1970)}
of some properties of mapping spaces, while
Section \ref{unnamedsec} introduces a language for universal constructions 
and Section \ref{partialadjoint} is devoted to a construction which
serves as a restricted sort of adjoint to a mapping space functor. 
Section \ref{moreadjoint} contains the construction of the functor which is
left adjoint to tensoring with the \calgs[] of compact operators $\compacts$.
After that  we proceed to analyze the \SEP
and operations which preserve this property in Section \ref{studysep}.
Section \ref{lifting} presents the slight alteration of the standard
lifting argument necessary for constructing the model category
structures we require.  Finally, in Section \ref{both}
we give the definition of the model category structure 
(Definition \ref{KK-model cat})
and state the main theorem (Theorem \ref{generalKK}).
In addition, Section \ref{both} also contains the definition of a
related model category 
structure in which the weak equivalences between \calgs\ are the
$K$-theory equivalences, which might be of independent interest.

Finally we would like to thank Ralf Meyer and Andreas Thom for various
interesting and useful discussions 
while completing this piece of work.

\section{Categorical Aspects of Topological *-Algebras}
\label{technical}

\begin{defn}[topological $*$--algebras and $\lmccstar$-algebras]\
\begin{enumerate}

\item  A topological \staralgs[ ] is an algebra over $\C$ equipped
  with a Hausdorff vector space topology such that
  the multiplication $A \times A \to A$ is separately continuous
  together with an involution $\ast:A \to A$.
  The involution is not required to be continuous. 
  Let $\staralg$ denote the category of topological
  \staralgs[s] together with the continuous \starhoms[].

\item  Suppose $S$ is a set of \cstar-seminorms on an (algebraic)
  \staralgs[ ]$A$.  Then there exists a coarsest
  topology $\tau_S$ on $A$ for which each of the seminorms in $S$ is continuous
  (compare \cite[I.3.1]{Mallios(1986)}). We write $A_S$ for the
  corresponding Hausdorff quotient of $A$,
  which then is an object in  $\staralg$.
  If a topology on $A$ is given which coincides with $\tau_S$
  for some set of $\cstar$-seminorms, it is called an
  $\lmccstar$-topology. Note that we do not require a
  \lmccstar-topology to be Hausdorff.
  The abbreviation $l.m.c.$ stands for ``locally multiplicative convex''.

\item  If $A$ is a topological \staralgs[ ] and $A \homeo A_S$ for
  some set of \cstar-seminorms $S$ then $A$ is called an \lmccstar-algebra.
  Equivalently, one could say that an \lmccstar-algebra is a topological \staralgs[]
  whose topology is given by a (Hausdorff) \lmccstar-topology.
  As in \cite[1.3.4 on page 9]{Dixmier(1977)},
  it follows that the involution of an \lmccstar-algebra is continuous.
  Let $\LMCCSTAR$ denote the full subcategory of
  \lmccstar-algebras in \staralg.

\item  For any $A \in \staralg$, let $S(A)$ denote the set of all
  continuous $\cstar$-seminorms on $A$. Thus, for any $A\in \LMCCSTAR$
  we have a canonical isomorphism $A \cong A_{S(A)}$.

\end{enumerate}
\end{defn}

We follow Mallios (cf.~\cite[page 484]{Mallios(1986)})
in the sense that we do not require an $\lmccstar$-algebra to be complete.
However, we would like to require some weak form of completeness
for the objects that we are going to work with.  Along with the added
convenience of some completeness,
this allows our tensor product to generalize the minimal tensor
product of \calgs.

\begin{defn}[$\nu$-complete algebras and the category $\ourcat$]\
\begin{enumerate}

\item
Let $\nu$ be an ordinal, whose cardinality is greater than or equal to the
cardinality of the integers.
A $\nu$-sequence in a topological $*$-algebra
simply is a map $x:\nu \to A$ of underlying sets. A $\nu$-sequence
is called a Cauchy sequence if, for any open neighborhood $U$ of $0\in A$, there
is an element $\gamma \in \nu$ such that for any two
$\alpha, \beta \in \nu$ satisfying $\alpha, \beta \ge \gamma$
one has $x(\alpha)-x(\beta)\in U$. An algebra $A\in \staralg$ is called
$\nu$-complete if each Cauchy $\nu$-sequence converges to a point in $A$.
We define the category $\ourcat$ to be the full subcategory
of $\staralg$ consisiting of the $\nu$-complete $\lmccstar$-algebras.
Moreover, let $\calg$ denote the full subcategory of $\staralg$
consisting of the \calgs.
Then we have inclusions of full subcategories
\[\calg \subset \ourcat \subset \LMCCSTAR \subset \staralg.
\]

\item For any $A\in \staralg$ one can define its $\nu$-completion as
  the expected quotient of the space of Cauchy $\nu$-sequences,
which we denote $\complete{A}$. As usual $\complete{A}$ comes equipped with a
canonical dense inclusion $A \subset \complete{A}$, which yields a
functor $\complete{?}: \LMCCSTAR \to \ourcat$. By contrast, the
(honest) completion of an 
algebra $A\in \staralg$ will be denoted $\overline{A}$ and we will use
the symbol $\ouriso$ to differentiate isomorphisms in the 
category $\ourcat$ from ordinary bijections (denoted $\cong$).
\end{enumerate}
\end{defn}

For ease of reference later, we state the following consequence of the idempotence of the
$\nu$-completion operation.

\begin{lemma} \label{completion_is_leftadjoint}
The functor $\complete{?}: \LMCCSTAR \to \ourcat, \ A \to \complete{A}$
is left adjoint to the inclusion $\ourcat \subset \LMCCSTAR$ of the
$\nu$-complete  \lmccstar-algebras into the category of all
\lmccstar-algebras,  i.e.~for any $B\in \ourcat$ we have a natural bijection
\[Hom(\complete{A}, B) \cong Hom(A,B).\]
\end{lemma}

Complete $\lmccstar$-algebras may be identified with inverse limits of
$\calgs$ by the Arens-Michael decomposition theorem
(see Theorem \ref{Arens-Michael} below).
Inverse limits of \calgs, sometimes called pro-\calgs, have been studied by
many authors, most notably by
Phillips (cf.~\cite{Phillips(1988a)}, \cite{Phillips(1988b)}).
For technical reasons that will become clear below
(Remark \ref{explain_why_lambda_completion_is_used}),
we are forced to work with the weaker notion of $\nu$-completeness.
Nevertheless, we can choose the ordinal $\nu$ to be
arbitrarily large.

\begin{lemma}
  \label{limits}
  The category $\ourcat$ has all set-indexed limits.
\end{lemma}

\begin{proof}
  Limits in $\staralg$ are formed as in topological spaces, with the
  algebraic operations defined coordinatewise
  (compare~\cite[III, Lemma 2.1]{Mallios(1986)}).
  Moreover,
  a limit in \staralg\ over a system of \lmccstar-algebras is an
  \lmccstar-algebra (compare \cite[III, (2.8)]{Mallios(1986)}).
  The limit in \staralg\ over a system
  of $\nu$-complete \lmccstar-algebras is also a
  $\nu$-complete \lmccstar-algebra, since a Cauchy $\nu$-sequence in a
  product or limit is simply a $\nu$-sequence that is Cauchy in each
  factor as in \cite[7.2 on page 57]{Kelley-Namioka(1963)}.
\end{proof}

\begin{lemma}
  \label{colimits}
  The category $\ourcat$ has all set-indexed colimits.
\end{lemma}

\begin{proof}
  First note that it suffices to construct
  colimits in $\LMCCSTAR$. Colimits in $\ourcat$ are
  then obtained by applying the $\nu$-completion
  to the corresponding colimit in $\LMCCSTAR$
  (by \ref{completion_is_leftadjoint}).
  Also, as commented above Theorem IX.1.1 in \cite{MacLane}, it suffices to
  construct coproducts, coequalizers (also called difference cokernels)
  and directed colimits in $\LMCCSTAR$.

  The existence of directed  colimits can be shown as
  in the proof of Lemma 2.2 in Section IV of \cite{Mallios(1986)}.
 
  Given two morphisms $u,v:A \to B$ in $\LMCCSTAR$, let $S(u,v)\subset S(B)$
  denote the set of continuous $C^*$-seminorms on $B$ which vanish on the
  ideal generated by elements of the form $u(a)-v(a)$ for $a \in A$.
  Then the desired coequalizer is given by the canonical map $B \to B_{S(u,v)}$.
   
  Suppose $\{A_i\}$ denotes a set of $\lmccstar$-algebras indexed
  on the set $I$.  We define a topology on the algebraic free product
  $\ast_i A_i$ by considering the collection of all \starhoms from the
  free product to $\lmccstar$-agebras $B$ such that the canonical
  algebraic compositions $A_i \to \ast_i A_i \to B$ are continuous for
  each $i \in I$. Given such a $\starhom$ we pull back the defining
  $\cstar$-seminorms of $B$ to the free product $\ast_i A_i$.
  As there is only a set of possible $\cstar$-seminorms on $\ast_i A_i$,
  this procedure yields a set $S$ of seminorms on $\ast_i A_i$ (even though
  the collection of \starhoms we are considering might not form a set).
  It is now easy to check that $(\ast_i A_i)_S$
  is the coproduct of the $A_i$ in $\LMCCSTAR$.
  (For a second approach see \ref{univ_examples} \eqref{free_product_in_G_and R}).
\end{proof}

  Let $A$ be a topological \staralgs[.]
  For any $p\in S(A)$, let $A_p$ denote the Hausdorff quotient
  $A_{\{p\}}$. Then $A_p$ is a
  pre-$\cstar$-algebra in the sense that its completion
  $\overline{A_p}$ 
  is a $\cstar$-algebra (here $\nu$-completion agrees with completion
  since the topology is first countable).
  Notice also that the set $S(A)$ is partially ordered by
  $p \leq q$ when $p(a) \leq q(a)$ for each $a \in A$, so there is a
  canonical
  map $A \to \lim A_p$, where the inverse limit is taken over the
  directed system $S(A)$.
  This map is shown to be an isomorphism for any complete
  $\lmccstar$-algebra by the remarkable Arens-Michael decomposition
  theorem.

\begin{thm}
  [\cite{Arens(1952)},\cite{Arens(1958)},\cite{Michael(1952)},
 {compare \cite[III.Theorem 3.1 on page 88]{Mallios(1986)}}]
  \label{Arens-Michael}
  For an $\lmccstar$-algebra $A$ there is a string of natural inclusions
  \[ A \subset \lim A_p \subset \lim \overline{A_p} \ouriso \overline{A}.\]
  In particular, for any complete $\lmccstar$-algebra there is a natural isomorphism
  \[ A \ouriso \lim A_p.\]
\end{thm}
 
  The Arens-Michael decomposition theorem is quite helpful in
  understanding the nature of \lmccstar-algebras in general.
  This already becomes apparent when
  we are considering the tensor products of \lmccstar-algebras, which
  we now introduce.

  As in the category of \calgs, there are various possibilities to
  define a reasonable tensor product in $\ourcat$. Recall that there is the
  notion of a maximal or a minimal tensor product
  in the category of \calgs, and both definitions agree if one factor is nuclear.
  By the tensor product of \calgs, we will always mean the minimal
  tensor product.

\begin{defn}\label{tensor_product}
  Let $A, B\in \ourcat$. Given $p\in S(A)$ and $q\in S(B)$ let
  $pq$ denote the $\cstar$-seminorm on the algebraic tensor product
  $A \tensor_{alg} B$ obtained by pulling back the
  $\cstar$-norm on $\overline{A_p} \tensor \overline{B_q}$ along the canonical
  homomorphism $A \tensor_{alg} B \to \overline{A_p} \tensor  \overline{B_q}$.
  Let $S(AB)$ be the set of $\cstar$-seminorms
  given by $\{ pq \ | (p,q) \in S(A) \times S(B)\}$. We define the
  tensor product $A \tensor B$ in $\ourcat$ to be the $\nu$-completion
  of the $\lmccstar$-algebra $(A \tensor_{alg} B)_{S(AB)}$.
\end{defn}

  By the Arens-Michael decomposition theorem,
  $(A \tensor_{alg} B)_{S(AB)}$ can also be
  identified with the Hausdorff quotient of the
  \staralgs[] $A \tensor_{alg} B$ equipped with
  the topology induced by the canonical homomorphism
  \[ A \tensor_{alg} B \to \prod_{pq\in S(AB)} \overline{A_p} \tensor \overline{B_q}.\]

\begin{remark}
  The tensor product described above gives $\ourcat$ a symmetric monoidal structure.
  The fact that we restrict to $\nu$-complete algebras ensures that the
  category of \calgs\ forms a symmetric monoidal subcategory of $\ourcat$.
\end{remark}

\section{On Mapping Spaces}
\label{mapping}

For any pair of Hausdorff spaces $X$ and $Y$, let $\cof[X,Y]$ denote the
space of continuous maps with the compact-open topology, which is then
a Hausdorff space as well.
Similary, for a pair of pointed Hausdorff spaces $X$ and $Y$, let $\ptcof{X,Y}$
denote the subspace of $\cof[X,Y]$ consisiting of the basepoint preserving maps.
Let $\cpt$ denote the category of compact spaces and
let $\ptcpt$ denote the category of pointed compact spaces.
Note that adding a disjoint basepoint yields an inclusion of
categories $\cpt \subset \ptcpt$. Moreover, given a pointed
Hausdorff space $Y$ and a compact space $X$ we have an obvious homeomorphism
$\cof[X,Y] \cong \ptcof{X_+,Y}$, where $X_+$ denotes adding a disjoint
basepoint to $X$.

Now notice that for $A$ and $B \in \ourcat$, one can consider
the set of continuous \starhoms $\Hom(A,B)$ as a topological space with the
compact-open topology and think of $0$ as the basepoint in each
algebra.  Thus, $\Hom(A,B)$ is a subspace of
$\ptcof{A,B}$, hence is a (pointed) Hausdorff space as well.
We will use the symbol $\compacts$ to denote the \calgs[] of compact operators
on a separable Hilbert space.

Given $X \in \cpt$ and $B \in \staralg$, it should be clear that
$\cof[X,B] \in \staralg$, where each algebraic operation is defined
pointwise, and similarly for the pointed mapping space
$\ptcof{X,B}$ for $X\in \ptcpt$.
However, we will need the following stronger result.

\begin{lemma}
  \label{cin}
  Suppose $X \in \cpt$ and $B \in \ourcat$, then
  $\cof[X,B] \in \ourcat$. Similarly, if $X\in \ptcpt$,
then $\ptcof{X,B}\in \ourcat$.
\end{lemma}

\begin{proof}
  First notice that the compact-open topology on $\cof[X,B]$ is the
  same as the topology of uniform convergence, since $X$ is compact.
  Given $p \in S(B)$, one has an associated $\cstar$-seminorm $p_X$ on
  $\cof[X,B]$ given by
\[p_X(f)=\underset{x \in X}{sup}\ p(f(x)).
\]
  In fact, the collection of these $\cstar$-seminorms $p_X$ on
  $\cof[X,B]$ determines the topology of uniform convergence.  To
  see this, notice the sets
\[U_{p,r,f}=\{g \in \cof[X,B]:\forall x \in X, p(f(x)-g(x))< r\}
\]
  with $p \in S(B)$, $r \in \R_{>0}$ and $f \in \cof[X,B]$ form a
  subbasis for the topology in both cases.  Thus, $\cof[X,B]$ is an
  \lmccstar-algebra. Since $B\in \ourcat$ is $\nu$-complete, so is
  $\cof[X,B]$ with this topology, thus $\cof[X,B] \in \ourcat$.

  For $X\in \ptcpt$, the algebra $\ptcof{X,B}$ is a closed $\nu$-complete
  subalgebra of $\cof[X,B]$, hence we also have $\ptcof{B,X}\in \ourcat$.
\end{proof}

The following lemma is provided here mainly for ease of reference later.

\begin{lemma} \label{take_tensor_K_out_of_cof}
 If $X \in \cpt$ and $B \in \ourcat$ is separable, then there is an
 isomorphism in $\ourcat$, 
$\cof[X, B] \tensor \compacts \ouriso \cof[X, B \tensor \compacts]$.
If $X\in \ptcpt$ and $B\in\ourcat$ is separable, we have a corresponding
isomorphism $\ptcof{X,B} \otimes \compacts \ouriso \ptcof{X, B \tensor \compacts}$.
\end{lemma}

\begin{proof}
Since by the previous lemma  $\cof[X,B \tensor \compacts]$ is $\nu$-complete,
the canonical injection
$\cof[X,B] \tensor_{alg} \compacts \to \cof[X,B \tensor \compacts]$ yields an
inclusion $\cof[X,B] \tensor \compacts \subset \cof[X,B\tensor
  \compacts]$ (compare \cite[Proposition 3.4]{Phillips(1988a)}).
Thus, we have to show that the map is surjective.

For pairs of natural numbers $(i,j)$, let $x_{ij}: \compacts \to \C, k \mapsto k_{ij}$
denote the family of linear maps which are defined through the equation
$k(e_i) =\sum_{i,j} k_{ij} e_j$ for some fixed orthonormal basis $\{e_i\}_{i\in \N}$
of the underlying separable Hilbert space that $\compacts$ is acting on.
Moreover, abusing notation, let $x^B_{ij}$ also
stand for the corresponding linear map $x^B_{ij}: B\tensor \compacts \to B\tensor \C \ouriso B$
that is obtained by tensoring the map $x_{ij}:\compacts \to \C$ with $id: B \to B$.
In addition, let $e_{ij}\in\compacts$ denote the compact operator which
maps a general vector $\sum_k h_k e_k$ (with the $h_k\in \C$) to $h_j e_i$.
Each element $b\in B\tensor \compacts$ then has a unique presentation
$b= \sum_{ij} x^B_{ij}(b) \tensor e_{ij}$.

Let $f:X \to B\tensor \compacts$ be a pointed continuous function
and define corresponding functions $f_{ij}\in \cof[X,B]$ by
$f_{ij}(x)= x^B_{ij}(f(x))$.
The sum $\sum _{ij} f_{ij} \tensor e_{ij}\in \cof[X,B] \tensor \compacts$
then is a pre-image for $f\in \cof[X,B \tensor \compacts]$. 

To obtain the pointed statement observe that evaluation at the basepoint
induces a natural transformation of short exact sequences
\[
\xymatrix{ 
0 \ar[d]\ar[r] & \ptcof{X,B} \tensor \compacts \ar[d]\ar[r] &
\cof[X,B] \tensor \compacts \ar[d]^{\ouriso} \ar[r]^-{ev \tensor id}& B \tensor \compacts\ar[d]^{\ouriso}\ar[r]& 0\ar[d]\\
0 \ar[r] & \ptcof{X,B \tensor \compacts}\ar[r] &
\cof[X,B \tensor \compacts] \ar[r]^-{ev}& B \tensor \compacts \ar[r]& 0
}
\] 
The pointed statement therefore follows from the unpointed statement. 
\end{proof}

Notice, in the previous proof we used the $\nu$-completion with
$\nu$ at least countable to see that the infinite sum
still lies in $\cof[X,B]\tensor \compacts$.

In order to prove our key result on mapping spaces,
we will first require another technical result
which comes from \cite{Steenrod(1967)} and
\cite{tomDieck-Kamps-Puppe(1970)}.

\begin{lemma}
  \label{steenrods}
  Suppose $X$ is compact, while $Y$ and
  $Z$ are Hausdorff.  Then there is a natural inclusion
\[\cof[Y,\cof[X,Z]] \to \cof[X,\cof[Y,Z]]
\]
  which is a homeomorphism if $Y$ is compactly generated.
\end{lemma}

\begin{proof}
  Define topological inclusions
  $\iota_X:\cof[X \times Y,Z] \to \cof[X,\cof[Y,Z]]$ by
\[(\iota_X(x))(y)=f(x,y)
\]
  and similarly
  $\iota_Y:\cof[X \times Y,Z] \to \cof[Y,\cof[X,Z]]$ by
\[(\iota_Y(y))(x)=f(x,y).
\] 
  These maps are inclusions as both $X$ and $Y$ are Hausdorff spaces
  (compare \cite[Satz (4.9) on page 88]{tomDieck-Kamps-Puppe(1970)}).
  In general, these inclusions are not surjective.
  However, $\iota_Y$ is surjective since $X$ is compact (see
  \cite[3.4.7]{Steenrod(1967)}) while $\iota_X$ is onto if $Y$, hence $X \times
  Y$, is compactly generated (see \cite[3.4.9]{Steenrod(1967)}).  Thus, in
  any case $\iota_X \circ (\iota_Y)^{-1}$ is a topological inclusion
  and it becomes a homeomorphism if $Y$ is compactly generated.
\end{proof}

We now have our key property of mapping spaces.

\begin{prop}
  \label{cadjoint}
  For $A$,$B \in \ourcat$ and $X \in \cpt$, there is a natural
  inclusion
\begin{equation} \label{bijection_cadjoint}
\Hom(A,\cof[X,B]) \to \cof[X,\Hom(A,B)]
\end{equation}
  which is a homeomorphism if $A$ is compactly generated as a
  topological space.  In particular, the inclusion is a homeomorphism
  if $A \in \calg$.

  In case $X\in \ptcpt$, the same is true 
  for the corresponding map of pointed mapping spaces
\begin{equation} \label{bijection_cadjoint_pointed}
\Hom(A,\ptcof{X,B}) \to \ptcof{X,\Hom(A,B)}.
\end{equation}
\end{prop}

\begin{proof}
  The statement for $\calgs$ follows from the general one because $A \in \calg$
  would imply $A$ is compactly generated as a normed, hence first
  countable, space (see \cite{Steenrod(1967)}).

  We use the notation of the previous proof and
  consider the case $X\in \cpt$ first.
  On the set-theoretic level, one easily checks that the restriction
  of $\iota_X \circ (\iota_A)^{-1}$ to the subspace
\[\Hom(A,\cof[X,B]) \subset \cof[A,\cof[X,B]]
\]
  factors through the subspace
\[\cof[X,\Hom(A,B)] \subset \cof[X,\cof[A,B]].
\]
  Hence we obtain a topological inclusion
\[\Hom(A,\cof[X,B]) \to \cof[X,\Hom(A,B)].
\]
  On the other hand, if $\iota_X \circ (\iota_A)^{-1}$ is surjective,
  it is also straightforward to check that the pre-image of
  $\cof[X,\Hom(A,B)]$ is precisely $\Hom(A,\cof[X,B])$, giving the
  expected homeomorphism.

  Now assume $X$ is pointed.
  Notice that $\ptcof{X,B} \subset \cof[X,B]$ consists of those maps
  which vanish on the basepoint $x_0 \in X$.  Thus, a $\starhom$
  $\psi:A \to \cof[X,B]$ factors through $\ptcof{X,B}$ precisely when
  $(\psi(a))(x_0)=0$ for each $a \in A$.  However, this is equivalent
  to saying that the adjoint map $X \to \Hom(A,B)$ sends the basepoint
  $x_0$ to the constant map on zero, which is the basepoint of
  $\Hom(A,B)$. Thus we obtain the pointed version from the
  unpointed statement simply by restriction. 
\end{proof}

\section{Universal objects defined through generators and relations}
\label{unnamedsec}

>From the Arens-Michael decomposition theorem, we see that
\lmccstar-algebras (resp.~$\nu$-complete \lmccstar-algebras)
are very closely related to inverse limits of \cstar-algebras, or pro-\calgs.
In \cite[1.3]{Phillips(1988b)} Phillips introduced
the concept of a weakly admissible set of generators and relations
and showed that a weakly admissible set of generators and relations
can be used to define a universal associated inverse limit of \cstar-algebras.
The corresponding concept works equally well for the category $\ourcat$. 
A pair $(G,R)$ of generators and relations in this context then simply consists
of a set $G$ and a collection $R$ of statements about the elements in $G$ which make sense for elements 
of a $\nu$-complete \lmccstar-algebra. A representation of $(G,R)$ in
a $\nu$-complete \lmccstar-algebra $A$ then consists of a set-valued map
$\varrho: G \to A$ such that the elements $\varrho(g)$ for $g\in G$
satisfy the relations $R$ in $A$.

\begin{defn}[{\cite[Definition 1.3.4]{Phillips(1988b)}}]
A set $(G,R)$ of generators and relations is called weakly admissible if
\begin{enumerate}
\item The zero map from $G$ to the zero \cstar-algebra is a
  representation of $(G,R)$.
\item If $\varrho$ is a representation of $(G,R)$ in a \cstar-algebra
  $A$, and $B$ is a 
\cstar-subalgebra of $A$ containing $\varrho(G)$, then $\varrho$ is a
  representation of $(G,R)$ in $B$.
\item If $\varrho $ is a representation of $(G,R)$ into 
      a $\nu$-complete \lmccstar-algebra $A$, and $\varphi:A \to B$ is 
      a continuous surjective \starhom\ onto a \cstar-algebra $B$,
      then $\varphi \circ \varrho$ is a representation of $(G,R)$ in $B$.
\item If $A$ is a $\nu$-complete \lmccstar-algebra, and $\varrho: G
  \to A$ is a function such that, for every $p\in S(A)$ the
  composition of $\varrho$ with the \starhom\ $A \to \overline{A_p}$
 (notation of Theorem \ref{Arens-Michael}) is a representation of
  $(G,R)$ in $\overline{A_p}$, 
  then $\varrho$ is a representation of $(G,R)$ in $A$.
\item If $\varrho_1, ..., \varrho_n$ are representations of $(G,R)$ in
      \cstar-algebras $A_1,...,A_n$, 
      then $g \mapsto (\varrho_1(g),...,\varrho_n(g))$ is a
      representation of $(G,R)$ in $A_1 \times ... \times A_n$.
\end{enumerate}
\end{defn}

\begin{prop}[{\cite[Proposition 1.3.6]{Phillips(1988b)}}]
\label{define_through_G_and R}
Let $(G,R)$ be a weakly admissible set of generators and
relations. Then there exists
a $\nu$-complete \lmccstar-algebra $\ourcat(G,R)$, equipped with a representation
$\varrho: G \to \ourcat(G,R)$
of $(G,R)$, such that, for any representation $\varsigma$ of $(G,R)$
into a $\nu$-complete \lmccstar-algebra $B$, there is 
a unique continuous \starhom\  $\varphi: \ourcat(G,R) \to B$
satisfying $\varsigma= \varphi \circ \varrho$.
\end{prop}

\begin{proof}
The proof carries over verbatim from the one given in \cite{Phillips(1988b)},
except that the \lmccstar-algebra $\ourcat(G,R)$ has to be defined
as the Hausdorff $\nu$-completion $F(G)_D$ of the free associative complex
$*$-algebra $F(G)$ generated by $G$ with respect 
to the set of seminorms $D$ (rather than being its Hausdorff completion).
\end{proof}

\begin{examples}\label{univ_examples}\
\begin{enumerate}
\item
Let $G$ be a set of generators and let $F(G)$ denote the free
associative complex \staralgs[] on the set $G$. The requirement
that a set map
$X \to F(G)$ from a topological space $X$ into  
$F(G)$ is continuous defines a set of weakly admissible relations.
(cf.~{\cite[Example 1.3.5 (4)]{Phillips(1988b)}})
\item
Given a $\nu$-complete \lmccstar-algebra $A$, define $G$ to be the set
$A$ and define the set of relations to be the set of algebraic
relations between the elements of $A$ together with the
relations that imply the obvious composite map $A \to G \subset F(G)$ is
continuous.
Then $(G,R)$ is weakly admissible and  $\ourcat(G,R) \ouriso A$.
(cf.~{\cite[1.3.7]{Phillips(1988b)}})
\item \label{free_product_in_G_and R}More generally, given a set
  $\{A_i\}_{i\in I}$ of $\nu$-complete \lmccstar-algebras 
take the union of the elements in the
$A_i$ as a set of generators $G$ and let $R$ be the set of all the
  algebraic relations in the individual 
$A_i$ together with the
relations which imply that each of the maps $A_i \to G$ is
  continuous.
This yields a set of weakly admissible relations and
$\ourcat(G,R)$ is a model for the coproduct of the $A_i$ in
  $\ourcat$.
\item  \label{require_continuity} Given a $\nu$-complete
  \lmccstar-algebra $A$, a topological space $X$ and a
set map $X \to A$ one can define a weakly admissible set of relations
  with $G=A$
and the set of relations $R$ given by all the algebraic relations of
  $A$ together with 
the relations which imply the two maps $A \to G \subset F(G)$ and $X
  \to G \subset F(G)$ to be continuous.
(This construction will be used in \ref{defining_atimes}.)
\end{enumerate}
\end{examples}

\section{Partial Adjoint Pairs}
\label{partialadjoint}

Ideally we would like to have a simple
and easy to work with
construction for a left adjoint to the functor
\[\righta=\Hom(A,?):\ourcat \to \pttop
\]
associated to any $A \in \ourcat$.  
However, technically it is easier
to deal with functors with the
following much weaker property, which suffices for our purposes.

\begin{defn}
  Suppose $\lefta:\ptcpt \to \ourcat$ is a functor together with a
  bijection
\[\Hom(\lefta(X),B) \to \ptcof{X,\righta(B)}
\]
  natural in both $X \in \ptcpt$ and $B \in \ourcat$.  Then $\lefta$
  will be referred to as a partial left adjoint to $\righta$.
\end{defn}

We now have the following relevant construction, which 
related to that of Phillips in \cite[on page 180]{Phillips(1988b)}.

\begin{construct} \label{defining_atimes}
  Suppose $A \in \ourcat$ and $X \in \cpt$.  For $x \in X$, let $A_x$
  denote a copy of $A$. By \ref{univ_examples} \eqref{free_product_in_G_and R} 
  we can obtain a model for the free product
  $\ast_{x \in X} A_x$ in $\ourcat$ in terms of generators and relations. 
  And by \ref{univ_examples} \eqref{require_continuity}
  we can put further relations
  on the free product which imply the natural inclusion
  $A \times X \to \ast_{x \in X} A_x$,
  given by the isomorphisms
  $A \times \{x\} \to A_x$, is continuous. 
  Let us denote the resulting object by $A \atimes X$. 

  By construction, $A \atimes X$
  comes equipped with a natural continuous map
\[i:A \times X \to A \atimes X.
\]
  In particular, for each $a \in A$ we obtain a continuous map
  $i_a:X \to A \atimes X$ by
\[X \cong \{a\} \times X \to A \times X \to A \atimes X
\]
  and similarly we obtain a continuous map $i_x:A \to A \atimes X$
  for each $x\in X$.

  If $x_0 \in X$ is a basepoint for $X$, then we can modify the
  previous construction to yield an \lmccstar-algebra $A \ptatimes X \in \ourcat$ and
  a corresponding continuous map
\[j:A \Smash X \to A \ptatimes X
\]
  simply by removing $x_0$ from the indexing set of the free product
  and requiring the canonical map $A \Smash X \to \ast_{x \in X \backslash \{x_0\}} A_x$
  to be continuous. As above, we then obtain continuous maps $j_a: X \to A \ptatimes X$ for
  each $a\in A$ and $j_x: A \to A \ptatimes X$ for each $x\in X$, with the maps $j_0$ and $j_{x_0}$
  mapping constantly to $0 \in A \ptatimes X$.
\end{construct}

\begin{prop}
  \label{halfadjoint}
  Let $A$,$B \in \ourcat$ and $X \in \cpt$.  Then there is a continuous
  bijection
\begin{equation}
\label{bijection_halfadjoint}\Hom(A \atimes X,B) \to \Hom(A,\cof[X,B])
\end{equation}
  natural in $A$,$B$ and $X$.
\end{prop}

\begin{proof}
  There is a continuous composition of
\[i^*:\Hom(A \atimes X,B) \to \cof[A \times X,B]
\]
  and the map
\[\iota_A:\cof[A \times X,B] \to \cof[A,\cof[X,B]]
\]
  which restricts to give the continuous map
\[\Hom(A \atimes X,B) \to \Hom(A,\cof[X,B])
\]
  in the statement.  Clearly, this map is natural in all variables, so
  it suffices to verify this is a bijection.

  In order to establish surjectivity, let $\psi$ be an arbitrary $\starhom$ in
  $\Hom(A,\cof[X,B])$.  As in the proof of Proposition \ref{cadjoint}, the
  map $\iota_A$ is a homeomorphism since $X$ is compact.  Let $f$
  denote the continuous map $f:A \times X \to B$ with
  $\iota_A(f)=\psi$ so $(\psi(a))(x)=f(a,x)$.  Then $f$ is a representation
  of the set of generators and relations defining $A \atimes X$.
  So, by Proposition \ref{define_through_G_and R}, we get a $\starhom$
  $g: A \atimes X \to B$ which then maps to $\psi$ by construction.

  In order to establish injectivity, note that two
  $\starhoms$ $g_1,g_2: A \atimes X \to B$ mapping to
  a $\starhom$ $\psi: A \to \cof[X,B]$ restrict to the same representation
  $f: A \times X \to B$, so we obtain $g_1=g_2$ by the uniqueness assertion 
  of Proposition \ref{define_through_G_and R}.
\end{proof}

\begin{prop}
  \label{almostadjoint}
  Let $A$, $B \in \ourcat$ and $X \in \ptcpt$.  Then there is a continuous
  bijection
\begin{equation}
\label{bijection_almostadjoint}
\Hom(A \ptatimes X,B) \to \Hom(A,\ptcof{X,B})
\end{equation}
  natural in $A$,$B$ and $X$.
\end{prop}

\begin{proof}
 There is a continuous composition of
\[j^*:\Hom(A \ptatimes X,B) \to \ptcof{A \Smash X,B}
\]
  and the map
\[\iota_A:\ptcof{A \Smash X,B} \to \ptcof{A,\ptcof{X,B}}
\]
  which restricts to give the continuous map
\[\Hom(A \ptatimes X,B) \to \Hom(A,\ptcof{X,B})
\]
  in the statement.  Clearly, this map is natural in all variables, so
  it suffices to verify this is a bijection, which follows as in the
  proof of Proposition \ref{halfadjoint}.
\end{proof}

In the next section, we will have a mild extension of the following
result which produces partial adjoint pairs.

\begin{prop}
  \label{partadjoint}
  For a compactly generated $A \in \ourcat$, the functor
\[A \ptatimes ?:\ptcpt \to \ourcat
\]
  is a partial left adjoint to the functor
\[\Hom(A,?):\ourcat \to \pttop.
\]
\end{prop}

\begin{proof}
  The statement is equivalent to the existence of a bijection
\[\Hom(A \ptatimes X,B) \to \ptcof{X,\Hom(A,B)}
\]
  natural in $X \in \ptcpt$ and $B \in \ourcat$.
  However, Proposition \ref{almostadjoint} yields a (continuous) bijection
\[\Hom(A \ptatimes X,B) \to \Hom(A,\ptcof{X,B})
\]
  natural in $X$ and $B$.
  For compactly generated $A$, Proposition \ref{cadjoint} yields a
  bijection (actually, a homeomorphism) natural in $X$ and $B$
\[\Hom(A,\ptcof{X,B}) \to \ptcof{X,\Hom(A,B)}.
\]
\end{proof}

\section{The adjoint of tensoring with the compact operators}
\label{moreadjoint}

>From the work of Phillips
(\cite[Proposition 5.8 on page 1084]{Phillips(1989)}),
we know that in the category of
inverse limits of \calgs\ there is an adjoint
to tensoring an object with the compact operators $\compacts$.
The corresponding statement also holds in the category $\ourcat$.
As with Phillips' construction, the adjoint can be obtained by an application
of Freyd's Adjoint Functor Theorem, but we instead give a constructive
proof.

Before we present the construction, we first have
the following lemma which verifies a necessary condition for
the existence of an adjoint to tensoring with  $\compacts$.

\begin{lemma}
Taking the tensor product with the compact operators in the category
$\ourcat$ commutes with
inverse limits $B = \lim_\beta B_\beta$ in $\ourcat$, i.e.~the
canonical map
\[ (\lim_\beta B_\beta) \tensor \compacts \stackrel{\ouriso}{\to}
lim_\beta (B_\beta \tensor \compacts)\]
is an isomorphism in $\ourcat$.
\end{lemma}

\begin{proof}
Using the Arens-Michael Decomposition Theorem \ref{Arens-Michael}
and the definition of the tensor product we obtain a commutative diagram
\[
\xymatrix{ 
(\lim_\beta B_\beta )\tensor \compacts \ar[r] \ar[d]& \lim_\beta (B_\beta \tensor \compacts) \ar[d]\\
\prod\limits_{q\in S(\lim_\beta B_\beta)} \overline{(\lim_\beta B_\beta)}_q \otimes \compacts \ar[r]& 
\prod\limits_\beta \prod\limits_{p\in S(B_\beta)} \overline{(B_\beta)}_p \otimes \compacts,
}
\]
where the vertical maps are inclusions. Now let $pr_{\alpha}: \lim_\beta B_\beta\to B_{\alpha}$
for an index $\alpha$ denote the canonical projection. The set $S=\bigcup_{\beta} \{ pr_\beta^* p\ |\ p\in S(B_\beta)\}$
of \cstar-seminorms on $\lim_\beta B_\beta$ then determines the topology of $\lim_\beta B_\beta$, 
which implies that the bottom horizontal arrow is an inclusion. 
Consequently so is the horizontal arrow on the top, which is the map of the lemma.
Thus, it remains to see that the \starhom\ is surjective.

For pairs of natural numbers $(i,j)$, let $x^B_{ij}: B \tensor
\compacts \to B$ be defined as in the proof of Lemma
\ref{take_tensor_K_out_of_cof}.
The \starhom\ of the lemma then fits into a commutative diagram
\[
\xymatrix{
(lim_\beta B_\beta) \tensor \compacts \ar[rr] \ar[dr]_{\prod
    x^{\lim_\beta B_\beta}_{ij}} & & lim_\beta (B_\beta \tensor
  \compacts) \ar[dl]^{\prod lim_\beta x^{B_\beta}_{ij}}\\
& \prod_{i,j} \lim_\beta B_\beta
}
\]
Given any element $b=(b_\beta)_{\beta}\in lim_\beta (B_\beta \tensor
\compacts)$, the element
$\sum_{i,j} x^{B_\beta}_{ij}(b_\beta) \tensor e_{ij}$ is contained in
$(\lim_\beta B_\beta) \tensor \compacts$
and provides a pre-image of $b$ (establishing surjectivity). Note that we have used that
$(\lim_\beta B_\beta) \tensor \compacts$
is sequentially complete, since the infinite sum above is not
contained in the algebraic tensor product.
\end{proof}

\begin{construct}
Suppose $A \in \ourcat$. For any
\lmccstar-algebra $Z$ and any \starhom\ 
$f: A \to Z \tensor \compacts$, define $Z_f \subset Z$ to be the 
topological \substaralg\ generated by the set
$\{f(a)_{ij}\ |\ i,j\in \N, a\in A\}$, 
whose elements are given by $f(a)_{ij}= x^{Z}_{ij}(f(a))$ (notation
from the previous proof).
Let $D$ be the set of all isomorphism classes of pairs
$(Z,f)$ where $Z$ is an \lmccstar-algebra and $f$ is a \starhom\
$f: A \to Z \tensor \compacts$ such that $Z=Z_f$.
Notice that $D$ is in fact a set, as there is only a set of
isomorphism classes of \lmccstar-algebras that can be generated by
the set $\N \times \N \times A$.
On the set $D$ there is a partial ordering defined as follows.
Pick representatives
$d=(Z_d,f_d)$ and define $d\ge e$ if there exits a \starhom\
$\varphi^d_{e}:Z_d \to Z_e$ such that
$f_e=(\varphi^d_e \tensor id_\compacts) \circ f_d$.
The partially ordered set $D$ is directed,
as for $d,e\in D$ one can define $f: A \to (Z_d \times Z_e)\tensor \compacts
\cong (Z_d \tensor \compacts) \times (Z_e \tensor \compacts)$ by
$f(a) = (f_d(a),f_e(a))$ and define $Z=Z_f$.
The topological \staralgs[] $Z$ then is an \lmccstar-algebra and
the \starhom\ $f$ can be regarded as a \starhom\ into $Z \tensor \compacts$,
with $(Z,f)$ bigger than $d$ and $e$ respectively by construction.
Now define $A \atensorK \compacts$ to be the $\nu$-completion
of the inverse limit $\lim\limits_{d\in D} Z_d$.
\end{construct}

\begin{prop}\label{prop_bijection_for_adjoint_of_K}
For $A,B \in \ourcat$ there is a natural continuous bijection
\begin{equation}\label{bijection_for_adjoint_of_K}
Hom(A \atensorK \compacts , B) \to  Hom(A, B \tensor \compacts)
\end{equation}
which induces an isomorphism on homotopy classes
\begin{equation}\label{bijection_for_adjoint_of_K(homotopyversion)}
[A \atensorK \compacts , B] \to  [A, B \tensor \compacts]
\end{equation}
\end{prop}

\begin{proof}
First notice the statement concerning homotopy classes follows by simply
considering $\cof[I,B]$ as well as $B$, exploiting the naturality.

By the previous lemma, taking the tensor product with the compact
operators $\compacts$ commutes with inverse limits.
Moreover, taking the tensor product with $\compacts$
induces a natural continuous map
$Hom(A,B) \to Hom(A \tensor \compacts, B\tensor \compacts)$.
Thus, there is a corresponding continuous map
$Hom(\lim_d Z_d, B) \to
Hom(\lim_d Z_d \tensor \compacts, B \tensor \compacts)$.
On the other hand, there is the canonical continuous \starhom\
$A\to \lim_d Z_d\tensor \compacts$,
whose composition with the projection to the factor
$Z_d\tensor \compacts$ corresponding
to the index $d= (Z_d, f_d:A \to Z_d\tensor \compacts)$
is the \starhom\ $f_d$.
The map \eqref{bijection_for_adjoint_of_K}
then is defined as the composition of the continuous maps
\[Hom(\lim_d Z_d, B) \to Hom(\lim_d Z_d \tensor \compacts, B \tensor
\compacts) \to Hom(A, B \tensor \compacts),
\]
where the second map is induced by precomposing a \starhom\ with the
canonical \starhom\ $A \to \lim_d Z_d \tensor \compacts$.
In particular, \eqref{bijection_for_adjoint_of_K} is continuous by
construction.

That the map described yields a bijection essentially
corresponds to an application of Freyd's Adjoint Functor Theorem
(\cite[Theorem 3.3.3 on page 109]{Borceux(1994a)}).
Instead of reducing to the abstract result, we give a direct proof.

To see surjectivity, let $f:A \to B \tensor \compacts$ be a
continuous \starhom. Let $Z$ be the topological \substaralg\ of $B$ which
is generated by the set
of elements $\bigcup_{i,j\in \N,a \in A} x^B_{ij}(f(a))$.
Then $f$ can be factored through $Z \tensor \compacts$,
so that $(Z,f)$ can be regarded as an index in $D$.
The \starhom\ $f:A \to B \tensor \compacts$
then is the image of the composite \starhom\
\[\lim_d Z_d \to Z \to B
\]
which is given by the projection corresponding to the
index $(Z,f)$ followed by the inclusion $Z \subset B$.

To see injectivity, we first show that any \starhom\ $\lim_d Z_d \to B$
is given by the composition of some projection $\lim_d Z_d \to Z_d$
followed by an inclusion $Z_d \subset B$.
Given any \starhom\ $\lim_d Z_d \to B$,
let $Z'$ be its topological image and apply $? \tensor \compacts$.
The corresponding \starhom\
$\lim_d Z_d \tensor \compacts \to B \tensor \compacts$
then factors as the composition
\[ \lim_d Z_d \tensor \compacts \to Z' \tensor \compacts \to B \tensor \compacts.\]
When we precompose this \starhom\ with the canonical \starhom\
$A \to \lim_d Z_d \tensor \compacts$ we obtain a \starhom\
$f: A \to B \tensor \compacts$. Define the index $d=(Z, f)$,
where $Z$ is the topological \staralgs[]
generated by the set $\bigcup_{i,j\in \N, a \in A} x^{Z'}_{ij}f(a)_{ij}$
where (abusing notation) the range of $f$ is restricted to
$Z' \tensor \compacts$.
It follows that the \starhom\
$\lim_d Z_d \tensor \compacts \to Z' \tensor \compacts$
can be factored
\[ \lim_d Z_d \tensor \compacts \to Z \tensor \compacts \to Z' \tensor \compacts,\]
where the first \starhom\ is induced by the projection $\lim_d Z_d \to Z$.
Since the composition $\lim_d Z_d \to Z'$ is surjective
and $Z \to Z'$ is injective, it follows that $Z \cong Z'$,
so $\lim_d Z_d \to B$ indeed is given by
the composition of a projection followed by an inclusion.
Moreover, the index $d$ as well as the corresponding inclusion
$Z \subset B$
are uniquely determined by the \starhom\ $f: A \to B \tensor \compacts$.
Thus, the map \eqref{bijection_for_adjoint_of_K} is injective.
\end{proof}

\begin{prop}\label{cont_aspects}
Let $A\in \calg$, $B\in \ourcat$, $X\in \cpt$ and $Y \in \ptcpt$.
Then there are natural homeomorphisms
\begin{equation}\label{a_natural_homeo}
Hom(A \atensorK \compacts, \cof[X,B]) \to
\cof[X, Hom(A \atensorK \compacts, B)],
\end{equation}
\begin{equation}\label{pta_natural_homeo}
Hom(A \atensorK \compacts, \ptcof{Y,B}) \to
\ptcof{Y, Hom(A \atensorK \compacts, B)},
\end{equation}
and the following two maps, given by composing a map
 with the map \eqref{bijection_for_adjoint_of_K},
\begin{equation}\label{therhvmap}
\cof[X, Hom(A \atensorK \compacts, B)] \to \cof[X, Hom(A, B \tensor \compacts)],
\end{equation}
\begin{equation}\label{therhvmap_pointed}
\ptcof{Y, Hom(A \atensorK \compacts, B)} \to \ptcof{Y, Hom(A, B \tensor \compacts)},
\end{equation}
 are continuous bijections.
\end{prop}

\begin{proof}
First, notice the statements for the pointed setting follow from the 
corresponding unpointed version, as in the proof of Proposition \ref{cadjoint}.
Also by Proposition \ref{cadjoint}, the map \eqref{a_natural_homeo}
is a topological inclusion, while the map \eqref{therhvmap} certainly is continuous.
Thus, we need to show that the two maps are bijections.

Consider the following commutative diagram,
where the unlabeled vertical arrows are induced by the
continuous bijections \eqref{bijection_for_adjoint_of_K}
and where the horizontal arrows are induced by the
maps \eqref{bijection_cadjoint} of Proposition \ref{cadjoint}

\[
\xymatrix{
Hom(A \atensorK \compacts, \cof[X,B]) \ar[d] \ar[r] &  \cof[X, Hom(A \atensorK \compacts, B)] \ar[dd]\\
Hom(A, \cof[X,B] \tensor \compacts) \ar[d]^{\stackrel{\eqref{take_tensor_K_out_of_cof}}{\cong}} & \\
Hom(A, \cof[X,B \tensor \compacts)] \ar[r] & \cof[X, Hom(A, B \tensor \compacts)].
}
\]
>From Proposition \ref{prop_bijection_for_adjoint_of_K}, it follows that the
first left vertical arrow is a bijection. Therefore, the composition of the two
left vertical arrows is also a bijection by Lemma \ref{take_tensor_K_out_of_cof}.
Proposition \ref{prop_bijection_for_adjoint_of_K} in addition implies that
the right vertical arrow is a injective map, and the
lower horizontal arrow is a bijection by Proposition \ref{cadjoint}.
It follows that the right vertical arrow is a bijection, therefore the
upper horizontal arrow must also be a bijection.
\end{proof}

\begin{prop} \label{tensorK_adjoint_is_weq}
For $A \in \calg$ and $B \in \ourcat$ the canonical bijection \eqref{bijection_for_adjoint_of_K}
\[Hom(A \atensorK \compacts, B) \to Hom(A, B \tensor \compacts)\]
is a weak homotopy equivalence.
\end{prop}

\begin{proof} Given a choice of a basepoint for
$Hom(A \atensorK \compacts , B)$
we regard $Hom(A \atensorK \compacts , B)$ as a pointed space, and
\eqref{bijection_for_adjoint_of_K} as a pointed map.
We need to show that the maps
\[ \pi_n(Hom(A \atensorK \compacts, B)) \to \pi_n (Hom(A, B \tensor \compacts))\]
are isomorphisms for any choice of a basepoint in
$Hom(A \atensorK \compacts , B)$.

This immediately follows from the previous proposition.
To see surjectivity let $f'\in \ptcof{S^n,Hom(A, B \tensor\compacts)}$ be a
pointed map. Consider the map \eqref{therhvmap} for $X=S^n$.
The pre-image $f\in \cof[S^n,Hom(A \atensorK \compacts, B)]$ of $f'\in\cof[S^n,Hom(A, B \tensor\compacts)]$
then is a pointed map representing a pre-image $[f]\in \pi_n(Hom(A \atensorK \compacts, B))$ of
$[f']\in \pi_n(Hom(A, B \tensor\compacts))$.

Injectivity follows similarly.
Let $f,g: S^n \to Hom(A \atensorK \compacts, B)$ be pointed maps
and assume that their images $f',g': S^n \to Hom(A, B \tensor \compacts)$
are pointed homotopic. Let $H': S^n \times I \to  Hom(A, B \tensor \compacts)$ be a
pointed homotopy between $f'$ and $g'$. Now consider the map
\eqref{therhvmap} for $X=S^n \times I$. The pre-image of $H'$
then yields a pointed homotopy $H: S^n \times I \to Hom(A \atensorK \compacts, B)$
between $f$ and $g$.
\end{proof}

We now show that for any $A \in \calg$
the functor
\[\leftaatensorK = (A \atensorK \compacts) \ptatimes ?: \ptcpt \to \ourcat\]
is a partial left adjoint to the functor
\[\rightaatensorK = \Hom(A \atensorK \compacts,?):\ourcat \to \pttop.
\]

\begin{prop}
\label{another_partial}
For $X\in \ptcpt$, $A \in \calg$ and $B\in \ourcat$, there is a natural continuous bijection
\[ Hom((A \atensorK \compacts) \ptatimes X, B) \cong \ptcof{X, Hom(A \atensorK \compacts, B)}.\]
\end{prop}
\begin{proof}
The bijection is obtained by composing the continuous bijection
\[ Hom((A \atensorK \compacts) \ptatimes X, B) \to Hom(A \atensorK \compacts, \ptcof{X,B})\]
provided by Proposition \ref{almostadjoint}
with the homeomorphism \eqref{pta_natural_homeo}.
\end{proof}

The statements of the following discussion 
(in particular those of Proposition \ref{reduce_to_pi-iso} and Proposition \ref{verify_cofibrant_replacement}) 
will be needed in order to
identify the homotopy classes of maps of the model categories we are going to define in
Section \ref{both}.
\begin{notation}\label{introduce_p_C}
Let $p: \C \to \compacts$ be a \starhom\
which maps the unit $1\in\C$ to a rank one projection.
For any $C\in \ourcat$, let
$p_{C}: C \ouriso C \tensor \C \to C \tensor \compacts$
be the map induced by tensoring $id_C$ with $p$,
and let $q_{C}: C \atensorK \compacts \to C$ be its adjoint.
\end{notation}

\begin{prop} \label{checking_cofibrant_replacement}
For any $A,B \in \ourcat$ composing with $q_{B}$ yields a natural isomorphism
\[ [A \atensorK \compacts, B \atensorK \compacts] \to [A \atensorK \compacts, B]. \]
\end{prop}

We need a few technical facts before we proceed to prove the Proposition.

\begin{lemma}\label{technical_stuff}\
\begin{enumerate}
\item \label{ST1}
The map $p_\compacts: \compacts \to \compacts \tensor \compacts$ is homotopic to an isomorphism.
\item \label{ST2}
Any two isomorphisms $\compacts \to \compacts \tensor \compacts$ are homotopic.
\item \label{addKLHS} For $B,B'\in \ourcat$, precomposing with $p_B: B \to B \tensor \compacts$ induces
a natural isomorphism $
  [B \tensor \compacts, B' \tensor \compacts] \to
  [B , B' \tensor \compacts]$.
\end{enumerate}
\end{lemma}

\begin{proof}
All three statements are well-known. Statements \eqref{ST1} and \eqref{ST2}
follow for example from \cite[Lemma 4.3]{Meyer(2000)}.
For the third statement, choose any isomorphism
$\phi: \compacts \tensor \compacts \to \compacts$.
Using \eqref{ST1} and \eqref{ST2}, one can show that
sending a \starhom\ $f:B \to B' \tensor \compacts$
to the \starhom\ $(id_{B'} \tensor \phi)\circ (f \tensor id_\compacts): B \tensor \compacts \to
B' \tensor \compacts \tensor \compacts \to B' \tensor \compacts$
induces an inverse to the map considered in \eqref{addKLHS}.
\end{proof}

\begin{proof}[Proof of Proposition \ref{checking_cofibrant_replacement}]
Use Lemma \ref{technical_stuff} and
Proposition \ref{prop_bijection_for_adjoint_of_K}
to see that there is a natural sequence of isomorphisms
\begin{eqnarray*}
[A \atensorK \compacts, B]
& \stackrel{\eqref{bijection_for_adjoint_of_K(homotopyversion)}}{\cong} & [A, B \tensor \compacts]\\
& \stackrel{({p_{B\tensor\compacts})}_*}{\cong} & [A, B \tensor \compacts \tensor \compacts]\\
& \stackrel{\eqref{bijection_for_adjoint_of_K(homotopyversion)}}{\cong} & [A \atensorK \compacts, B \tensor \compacts]\\
& \stackrel{\ref{technical_stuff} \eqref{addKLHS}}{\cong} & [(A \atensorK \compacts) \tensor \compacts, B \tensor \compacts].
\end{eqnarray*}
Thus, in order to prove Proposition \ref{checking_cofibrant_replacement}
it suffices to show that the corresponding map
\[ [(A \atensorK \compacts) \tensor \compacts, ( B \atensorK
  \compacts) \tensor \compacts] \to 
[(A \atensorK \compacts) \tensor \compacts, B \tensor \compacts]
\]
is an isomorphism.

To see this, consider the map
\begin{equation} \label{intermediate_iso}
  [ A' \tensor \compacts, B' \tensor \compacts] \to [(A' \atensorK \compacts) \tensor \compacts, B' \tensor \compacts]
\end{equation}
induced by precomposing with
$q_{A'} \tensor id_\compacts: (A' \atensorK \compacts) \tensor \compacts \to A' \tensor \compacts$.
It is a natural isomorphism since, as above, it is given through a composition of isomorphisms
\begin{eqnarray*}
[ A' \tensor \compacts, B' \tensor \compacts]
& \cong & [ A', B' \tensor \compacts] \\
& \cong & [ A', B' \tensor \compacts \tensor \compacts] \\
& \cong & [ A' \atensorK \compacts , B' \tensor \compacts] \\
& \cong & [(A' \atensorK \compacts) \tensor \compacts, B' \tensor \compacts].
\end{eqnarray*}
If we take $A'=B'=B$ then, by construction, \eqref{intermediate_iso}
maps the homotopy class $[id_{B \tensor \compacts}]$ to $[q_B \tensor id_\compacts]$.
On the other hand, if we take $A'=B$ and $B'=B \atensorK \compacts$
the isomorphism guarantees that we can pick a map
$g:B \tensor \compacts \to (B \atensorK \compacts) \tensor \compacts$
such that the composition $g \circ f$ with $f=q_B \tensor id_\compacts$ is homotopic to
 $id_{(B \atensorK \compacts) \tensor \compacts}$.
The map $g$ then is a homotopy inverse to $f$, since
the map \eqref{intermediate_iso} for $A'=B'=B$ maps
$[f \circ g]$ to $[f \circ g \circ f] = [f \circ id_{(B \atensorK \compacts) \tensor \compacts}]=[f]$,
and therefore $[f \circ g]=[id_{B \tensor \compacts}]$, using the injectivity of \eqref{intermediate_iso}.
It follows that composing with $g$ induces an isomorphism $[A', B \tensor \compacts] \to [A', (B \atensorK
  \compacts) \tensor \compacts]$ for any $A'\in \ourcat$, 
in particular for $A'=(A \atensorK \compacts) \tensor \compacts$.
\end{proof}

\begin{defn}
 \label{def_of_qA}
  For $A\in \ourcat$, let $qA$ be the kernel of the
  fold map $A \ast A \to A$.
\end{defn}

In order to prove the following two propositions
below we first need to recall the well-known fact 
that the space $Hom(qA, B \tensor \compacts)$ is a group-like $H$-space, for any choice of  $B\in\ourcat$.
The $H$-space structure is induced by the $H$-space structure of $\compacts$; and 
the structure is group-like with homotopy inverse given by the
automorphism on $qA$ induced by switching the two factors of $A \ast A$.
For a reference see \cite[Proposition 1.4]{Cuntz(1987)}, which covers the case
where $B$ is a \calgs[;] the general case then follows using the Arens-Michael decomposition theorem.

Also recall that a map $X \to X'$ of group-like $H$-spaces is a weak homotopy equivalence 
if and only if it is a $\pi_*$-isomorphism (with the existing basepoint). 

\begin{prop}\label{reduce_to_pi-iso}
Let $A\in\calg$, and let $B \to B'$ a morphism in $\ourcat$. Then the induced map
\begin{equation}
Hom(qA \atensorK \compacts, B) \to Hom(qA \atensorK \compacts, B')
\end{equation}
is a weak equivalence if and only if the map 
$Hom(qA, B \tensor \compacts) \to Hom(qA, B' \tensor \compacts)$
is a $\pi_*$-isomorphism.
\end{prop}

\begin{proof} This immediately follows from Proposition \ref{tensorK_adjoint_is_weq}
and the fact that the map $Hom(qA, B \tensor \compacts) \to Hom(qA, B' \tensor \compacts)$
is a map of $H$-spaces.
\end{proof}

\begin{prop}\label{verify_cofibrant_replacement}
For $A\in\calg, B\in \ourcat$ the map induced by composing with $q_B: B \atensorK \compacts \to B$ 
(defined in \ref{introduce_p_C})
\[Hom(qA \atensorK \compacts, B \atensorK \compacts) \to Hom(qA \atensorK \compacts,B)\]
is a weak homotopy equivalence.
\end{prop}

\begin{proof}
By the previous proposition it suffices to show that the 
corresponding map $Hom(qA, (B \atensorK \compacts) \tensor \compacts) \to 
Hom(qA, B \tensor \compacts)$ is a $\pi_*$-isomorphism.

Note that for any $Y\in \ptcpt$ and a general $B' \in \ourcat$ 
we have the following natural composition of bijections
\begin{multline*}
     Hom((qA \ptatimes Y)\atensorK \compacts, B')
\stackrel{\eqref{bijection_for_adjoint_of_K}}{\cong}  Hom(qA \ptatimes Y, B' \tensor \compacts)\\
\stackrel{\eqref{bijection_almostadjoint}}{\cong}  Hom(qA, \ptcof{Y,B' \tensor \compacts})
\stackrel{\eqref{bijection_cadjoint_pointed}}{\cong}  \ptcof{Y,Hom(qA, B' \tensor \compacts)}.
\end{multline*}
Using naturality in $Y$ it follows that
for any $n\ge 0$ there is a canonical bijection
\[ [(qA \ptatimes S^n) \atensorK \compacts ,B' ] \cong [S^n, Hom(qA, B' \tensor \compacts)]_R,\]
where the right hand side denotes right homotopy classes of maps, i.e.~two pointed maps
$f,g:S^n \to Hom(qA, B' \tensor \compacts)$ are equivalent if there is a 
corresponding pointed map $H: S^n \to Hom(qA, \cof[I,B' \tensor \compacts])$ such that the maps
$S^n \to Hom(qA, B' \tensor \compacts)$ obtained from $H$ by composition with the evaluation maps
$\cof[I,B' \tensor \compacts] \to B' \tensor \compacts$ at $0$ and $1$ respectively are
the maps $f$ and $g$. 
However, using Lemma \ref{steenrods} and Proposition \ref{cadjoint} one gets
$\pi_n(Hom(qA, B' \tensor \compacts)) \cong [S^n, Hom(qA, B' \tensor \compacts)]_R$.
Thus, taken together we have canonical isomorphisms
\[ \pi_n(Hom(qA, B' \tensor \compacts)) \cong [(qA \ptatimes S^n)\atensorK \compacts, B'].\]
The statement of the proposition then follows from Proposition
\ref{checking_cofibrant_replacement}, applied to the $qA \ptatimes S^n$ for the various $n\in \N$.
\end{proof}

\section{The Seminorm Extension Property}
\label{studysep}

The point of this section is to introduce a class of maps which will
play roughly the role of embeddings in topological categories.

\begin{defn}
Let $f: A\to B$ be a continuous \starhom\ between 
\lmccstar-algebras. We say
$f$ has the seminorm lifting property if for any continuous
sub-multiplicative \cstar-seminorm $p\in S(A)$ there is a continuous
sub-multiplicative \cstar-seminorm $q\in S(B)$ such that $p=f^*q$.
In other words $f$ has the semi-norm extension property if and only if
$S(A) = f^*S(B)$.
\end{defn}

First, we have the key property of maps with the \SEP[], which will
allow the so-called small object argument to proceed in Section
\ref{lifting}.

\begin{lemma}
        \label{allsmall}
        All objects in $\LMCCSTAR$ are small with respect to
        \starhoms which have the \SEP[].
\end{lemma}

Before giving a proof of Lemma \ref{allsmall}, we need a better
understanding of the \SEP[].

\begin{lemma}
  \label{sepincl}
        Let $A,B$ be \lmccstar-algebras.
        If a \starhom\ $f: A\to B$ has the \SEP[,]
        then $f$ is a topological inclusion.
\end{lemma}

\begin{proof}
  If $x \neq y \in A$, then there exists $p \in S(A)$ with
  $p(x-y)\neq0$ because $A$ carries a Hausdorff \lmccstar-topology.
  Since $f$ has
  the \SEP[], there exists $q \in S(B)$ with $f^*q=p$.  Thus,
  $qf(x-y)=p(x-y)\neq0$ so $f(x)\neq f(y)$ in $B$, or $f$ is an
  injection.

  To see that $f$ is an inclusion, note that $A$ equipped with the
  topology pulled back from $B$ over $f$ will be $A_{f^*S(B)}$.
  However, since $f$ has the \SEP we have $f^*S(B)=S(A)$ which implies
  $A_{f^*S(B)}=A_{S(A)}=A$, the last equality from the fact that $A$
  is an \lmccstar-algebra.
\end{proof}

\begin{lemma}
        \label{embedcolim}
        Let $\lambda$ be an ordinal and suppose
        \[A_0\to A_1\to A_2 \to\dots A_\alpha \to\dots\]
        is a $\lambda$-sequence\footnote{For a definition of
          ``$\lambda$-sequence into a cocomplete category'' we refer the reader to
          \cite[2.1.1.]{Hovey}}
        over \starhoms\
        $A_\alpha \to A_{\alpha+1}$ in $\LMCCSTAR$ which have the
        \SEP[].  Then each
        $A_\alpha \to \underset{\alpha < \lambda}{\colim} A_\alpha$ has the
          \SEP and the
        underlying set of the colimit is the union of the underlying
        sets of the $A_\alpha$.\label{nicecolim}
\end{lemma}

\begin{proof}
  By the previous lemma, all structure maps in the colimit are
  (topological) inclusions. Hence the colimit of the sequence in the category
  of associative complex \staralgs\ is given by the union
  $A=\union_{\alpha<\lambda} A_\alpha$. On this union, there is a
  coarsest \lmccstar-topology and the colimit in \LMCCSTAR\ then is given
  by taking the Hausdorff quotient of $A$ with respect to this topology.
  We now show that the topology on the union $A$ is already Hausdorff.

  We claim that any two distinct points in $A$ are separated by a
  seminorm.  To see
  this, let $x \in A_{\alpha'}$ and $y \in A_{\alpha}$ denote two
  distinct points of the union.  Without loss of generality, assume
  $\alpha \geq \alpha'$ so $x \in
  A_\alpha$ as well.  Since $A_\alpha$ is a Hausdorff $\lmccstar$-algebra,
  there is a $\cstar$-seminorm $p \in S(A_\alpha)$ with $p(x-y)\neq0$. 
  We extend this to $q \in S(A)$ by transfinite induction.
  By assumption,
  we can extend any $p_\alpha \in S(A_\alpha)$ to
  $p_{\alpha+1} \in S(A_{\alpha+1})$ and
  this handles successor ordinals.  Now, for a limit ordinal $\beta <
  \lambda$, we
  define a seminorm on the union $\union_{\alpha < \beta} A_\alpha$
  using the
  universal property of the union with respect to continuous maps and
  the continuous $\cstar$-seminorms already defined.  This yields a
  continuous map $\union_{\alpha < \beta} A_\alpha \to \R$ which acts
  as a seminorm since each pair
  of points in the union lies in some $A_\alpha$ as above, thereby completing
  the transfinite induction. 

  We now have $q \in S(A)$ with $q(x-y)=p(x-y)\neq0$ for any pair of
  distinct points in $A$.
  Hence, $A_{S(A)}$ is a Hausdorff space, so taking the Hausdorff quotient is
  not necessary in constructing the colimit in $\ourcat$.  This
  implies the underlying set of the colimit in $\ourcat$ is precisely
  $A$, the union of the sets $A_\alpha$.  It also implies the
  extension $q \in S(A)$ of $p \in S(A_\alpha)$ remains a continuous
  $\cstar$-seminorm
  on the colimit in $\ourcat$, so each
\[A_\alpha \to \underset{\alpha < \lambda}{\colim} A_\alpha
\]
  has the \SEP[].
\end{proof}

We can now give a proof of Lemma \ref{allsmall}.

\begin{proof}[Proof of Lemma \ref{allsmall}]
        By the combination of Lemmas \ref{sepincl} and \ref{embedcolim}
        each of the morphisms in a $\lambda$-sequence
        $A_\alpha \to \underset{\alpha < \lambda}{\colim} A_\alpha$
        is a topological inclusion. Hence it suffices to see that
        any object $B \in \ourcat$, regarded as a set, is small with
        respect to set inclusions and that
        the underlying set of such a colimit is the colimit of the
        underlying sets.  The former fact is well known (see
        \cite[Example 2.1.5]{Hovey}), while the
        latter is part of Lemma \ref{nicecolim}.
\end{proof}

We can now verify the version of Lemma \ref{allsmall} we will use to
perform our small object arguments.

\begin{lemma}
        \label{allsmall_in_ourcat}
        All objects in $\ourcat$ are small with respect to
        \starhoms which have the \SEP[].
\end{lemma}

\begin{proof}  
        Let $\lambda$ be an ordinal and suppose
        \[A_0\to A_1\to A_2 \to\dots A_\alpha \to\dots\]
        is a $\lambda$-sequence
        over \starhoms\
        $A_\alpha \to A_{\alpha+1}$ in $\ourcat$ which have the
        \SEP[]. Let us write $C$ for the colimit over this sequence
        taken in the category \LMCCSTAR. The colimit in $\ourcat$ then
        is the $\nu$-completion $\complete{C}$ of $C$.
        In particular, the maps from $A_\alpha$ into the
        colimit $\complete{C}$ are all inclusions.
       
        The elements of $\complete{C}$ by definition can be represented by
        Cauchy $\nu$-sequences of elements of $C$.
       
        For any \starhom\ $f:B \to \complete{C}$ and $b\in B$
        let $(f_{b,i})_{i\in\nu}$ be a Cauchy sequence in $C$ which represents
        $f(b)$. By Lemma \ref{embedcolim}, for each pair $(b,i)$
        we can choose an ordinal $\alpha_{b,i} < \lambda$
        such that $f_{b,i}\in A_{\alpha_{b,i}}$.
        Now assume that $\lambda$ is $|B| \times |\nu|$-filtered
        (see \cite[Definition 2.1.2]{Hovey}).
        Let $\gamma= \sup\{ \alpha_{b,i}\ |\ (b,i)\in B \times \nu\}$, so
        $\gamma < \lambda$ and $f(B) \subset A_\gamma$. Since $f$ was arbitrary, we conclude
        $Hom(B, \complete{C}) \cong \colim_\alpha Hom(B, A_\alpha)$.
\end{proof}

\begin{remark}\label{explain_why_lambda_completion_is_used}
The previous lemma is the reason why we work with $\nu$-complete
\lmccstar-algebras instead of complete \lmccstar-algebras.
Roughly speaking, it was crucial in the proof that we
have an a priori set theoretic bound for the length of
the Cauchy nets we are considering for the completion.
\end{remark}

To apply Lemma \ref{allsmall}, we
will need to know certain formal constructions preserve the \SEP[].
Lemma \ref{nicecolim} already says transfinite compositions preserve
the \SEP and the first lemma below gives a left cancellation result.

\begin{lemma}
        \label{firstfactor}
        Suppose $f:A \to B$ and $g:B \to C$ are continuous \starhoms[].
        Then $gf:A \to C$ having the \SEP implies
        that $f$ has the \SEP[].
\end{lemma}

\begin{proof}
  Suppose $p \in S(A)$ and $q \in S(C)$ with $(gf)^*q=p$.  Then
  $f^*g^*q=p$ so $g^*q \in S(B)$ is the required extension of $p$ over
  $f$.
\end{proof}

\begin{lemma}
        \label{SEPcobase}
       Let
\[
\xymatrix{B & \ar[l]_{f} A \ar[r]^{g} & C
}
\]
       be a diagram in $\ourcat$ with $f:A\to B$
       having the \SEP[]. Then the canonical map $h:C \to B \ast_A C$
       has the \SEP[].
\end{lemma}

\begin{proof}
  Suppose $p \in S(C)$ and let $r=g^*p \in S(A)$ with $q \in S(B)$ an
  extension of $r$ over $f$.  Now let $C_1$ denote the completion of
  $C$ with respect to $p$ (or $\overline{C_p}$ from below Lemma
  \ref{colimits}), $B_1$ the completion of $B$ with respect to
  $q$ and $A_1$ the completion of $A$ with respect to $r$.  This
  implies $A_1,B_1$ and $C_1$ are $\calgs$ and the completed versions,
  $f_1$ and $g_1$ of $f$ and $g$ respectively, have become isometries.
  Now take the pushout in the category of $\calgs$,
\[
\xymatrix{
A_1 \ar[r]^{g_1} \ar[d]_{f_1} & C_1 \ar[d]^{h_1} \\
B_1 \ar[r] & P
}
\]
  so that $P$ represents the appropriate quotient of the amalgamated
  free product of $\calgs$.  Notice we have a commutative diagram of
  continuous \starhoms
\[
\xymatrix{
A \ar[r]^{g} \ar[d]_{f} & C \ar[d]^{h} \ar[r] & C_1 \ar[dd]^{h_1} \\
B \ar[r] \ar[d] & B \ast_A C \\
B_1 \ar[rr] && P
}
\]
  hence an induced continuous $\starhom$ from the universal property of
  $B \ast_A C$ which makes the diagram
\[
\xymatrix{
C \ar[r]^{g_1} \ar[d]_{h} & C_1 \ar[d]^{h_1} \\
B \ast_A C \ar[r] & P
}
\]
  commute.  Since $f_1$ and $g_1$ are isometries, the
  induced $\starhom$ $h_1:C_1 \to P$ is an isometry as well (see
  \cite{Mallios(1986)}).  In particular, the $\cstar$-norm of $P$ extends
  that of
  $C_1$, so commutativity of the last diagram implies pulling back the
  $\cstar$-norm on
  $P$ to a $\cstar$-seminorm on $B \ast_A C$ extends the original
  $\cstar$-seminorm $p \in S(C)$ as required.
\end{proof}

Recall the construction of the coproduct in $\ourcat$ as described in
\ref{univ_examples} (\ref{free_product_in_G_and R}),
written with the usual free product notation.

\begin{cor}
        \label{SEPcoprod}
       Let $I$ be a set, and let
       $\{f_i\}_{i \in I}$ be a
       set of morphisms $f_i: A_i \to B_i$ in $\ourcat$
       which all have the \SEP[].  Then the canonical map
       $\ast_{i\in I} f_i:
       \ast_{i\in I} A_i \to \ast_{i\in I} B_i$
       of coproducts in $\ourcat$ has the \SEP[].
\end{cor}

\begin{proof}
  Suppose $f:A \to B$ and $g:C \to D$ are morphisms in $\ourcat$
  which have the \SEP[].  Notice $B \ast C \approx B \ast_A (A \ast
  C)$, so the diagram
\[
\xymatrix{
B & \ar[l]_{f} A \ar[r] & A \ast C
}
\]
  and Lemma \ref{SEPcobase} imply $A \ast C \to B \ast C$ has the
  \SEP[].  Since the symmetric argument implies $B \ast C \to B \ast
  D$ also has the \SEP, and composition clearly preserves the \SEP[],
  we conclude that $A \ast C \to B \ast D$ has the \SEP[]. 

  The previous paragraph is the successor ordinal case and suggests
  the general case may be considered as a transfinite composition
  after any choice of ordering for $I$.
  Hence, the limit ordinal case is handled by Lemma \ref{nicecolim}
  and the claim follows by transfinite induction.
\end{proof}

The key to applying the results on the \SEP to the expected model
structures is the next proposition.
We'll take zero as the basepoint of $I$ in general.

\begin{defn}
Let $A$ be an algebra in $\proa$.
We say $A$ has the \CSEP if the canonical
morphism $i_1: A\to A \ptatimes I$ 
has the \SEP[]. If $i_1: A \ptatimes X \to (A \ptatimes X) \ptatimes I$ has
the \SEP for any $X \in \ptcpt$ we will say that $A$ has the \sCSEP[].
\end{defn}

\begin{remark}
  The proof of the following proposition relies upon a result in
  $KK$-theory.  Namely, for any separable $A \in \calg$, any Hilbert space
  ${\mathcal H}$ and any $X \in  \ptcpt$ one has $KK(A,\ptcof{X,\bofh})=0$.
  This follows from an Eilenberg-swindle type of argument.
\end{remark}

For the next statement recall from Definition \ref{def_of_qA} that $qA$ for an $A\in\ourcat$
denotes the kernel of the fold map $A\ast A \to A$.
\begin{prop}
  \label{qgood}
  If $A$ is a separable \calgs[], then $qA \atensorK \compacts$ has the \sCSEP[].
\end{prop}

\begin{proof}
  For ease of reference, let $B=(qA \atensorK \compacts) \ptatimes X$
  and suppose 
  $p \in S(B)$.  We must show
  there is an extension to $q \in S(B \ptatimes I)$ over the map
  $i_1:B \to B \ptatimes I$.
  Clearly, there
  is a composite morphism in $\ourcat$
\[\psi:B \to \overline{B_p} \to \bofh
\]
  with $p$ the pullback under $\psi$ of the norm on $\bofh$ (since
  $\overline{B_p}$ is a $\calg$). 
  Using Proposition \ref{almostadjoint}  and
  Proposition \ref{prop_bijection_for_adjoint_of_K}, this corresponds to
  a unique $\starhom$ between $\calgs$
\[\psi':qA \to \ptcof{X,\bofh} \tensor \compacts.
\]
  Since $KK(A,\ptcof{X,\bofh})$ is trivial,
  Cuntz's Theorem \ref{htpyKK}
  implies that $\psi'$ represents the zero element
  in the group.  Hence, there is a null homotopy
\[H:qA \to \ptcof{I, \ptcof{X,\bofh} \tensor \compacts}
\]
  of \starhoms[{}].
  Recall that there is an isomorphism of $\calgs$
\[\ptcof{I, \ptcof{X,\bofh} \tensor \compacts}
    \approx \ptcof{X,\ptcof{I,\bofh \tensor \compacts}}
\]
  since both $X$,$I \in \ptcpt$.  Using Proposition \ref{almostadjoint}  and
  Proposition \ref{prop_bijection_for_adjoint_of_K} again, $H$ yields a corresponding \starhom\
\[H': B \ptatimes I \to \bofh
\]
  which factors $\psi$ as follows
\[
\psi: B \stackrel{i_1}{\to} B \ptatimes I  \stackrel{H'}{\to}
 \bofh.
\]
  Hence, the pullback of the norm on $\bofh$ along $H'$ yields
  the desired $q \in S(B \ptatimes I)$ extending $p$ over $i_1$.
\end{proof}

\section{Building Model Structures on $\ourcat$}
\label{lifting}

A standard technique in model categories involves building a new model
structure from a known model structure together with an adjoint pair.
This technique is referred to as lifting, and results saying when such
operations are successful are generally referred to as lifting lemmas.

For the basic object(s) $A \in \ourcat$, we will make the following
two assumptions throughout this section.
\begin{enumerate}
\item   $\righta=\Hom(A,?)$ has $\lefta=A \ptatimes ?$ as a partial left adjoint
\item   there is a natural homeomorphism
  \[\righta(\ptcof{X,B}) \to \ptcof{X,\righta(B)}
  \]
\end{enumerate}
 
We would like to give a lifting lemma under these assumptions, so we
should begin by pointing out that we have a variety of objects which
satisfy both assumptions. 
We know that for any separable $A' \in \calg$, the object
$A=qA' \atensorK \compacts$ satisfies assumption (1) by
Proposition \ref{another_partial} and
assumption (2) by Proposition
\ref{cont_aspects} equation \eqref{pta_natural_homeo}.
Near the end of this section, we will also need to assume $A$ has the
\sCSEP[] in order to
produce our model structures, which holds for
$A=qA' \atensorK \compacts$ by Proposition \ref{qgood}.

Although this lifting technique is standard, the proofs
are provided in detail because we do not quite consider
an adjoint pair.  (Compare with \cite[Section 8]{Michele}.)
The following is not particularly surprising but is a technical
necessity for the argument that follows.

\begin{lemma}\ 
\label{htpy}\label{nullhtpy}
\begin{enumerate}
\item   Suppose $A$ satisfies our second assumption.  Then
   for any $B \in \ourcat$ and $X \in \ptcpt$, there is a natural map
\[\righta (B) \Smash X \to \righta(B \ptatimes X).
\] 
        In fact, this
        is given by $f \Smash x \longmapsto i_x \circ f$.

\item   For any $B \in \ourcat$ and $X$,$Y \in \ptcpt$,
        there is a natural \starhom
 \[ (B \ptatimes X) \ptatimes Y \to  B \ptatimes (X \Smash Y).
\]
       
\end{enumerate}
\end{lemma}

\begin{proof}
        Recall from the construction of $\ptatimes$ that there is a
        natural pointed continuous map $j: B \Smash X \to B \ptatimes X$. 
        Since $X$ is compact, it follows that
\[\iota_B:\ptcof{B \Smash X,B \ptatimes X} \to
           \ptcof{B,\ptcof{X,B \ptatimes X}}
\]
        is a homeomorphism, as in the proof of Lemma \ref{steenrods}.
        Thus, the adjoint of $j$,
\[\iota_B(j): B \to \ptcof{X,B \atimes X}
\]
        is
        pointed, continuous and compatible with the algebraic structure,
        i.e.~it is an element in $\Hom(B,\ptcof{X, B \atimes X})$.
        This induces a pointed continuous map (by postcomposition)
\[\righta(B) \to \righta(\ptcof{X,B \atimes X}).
\]
        By 
        our second assumption on $A$, we have a natural homeomorphism
\[\righta(\ptcof{X, B \ptatimes X}) \to
        \ptcof{X,\righta(B \ptatimes X)}
\]
        hence, an induced pointed continuous map       
\[\righta(B) \to \ptcof{X,\righta(B \ptatimes X)}.
\]
        Finally,
\[\iota_{\righta(B)}:\ptcof{\righta(B) \Smash X,\righta(B \ptatimes X)} \to
           \ptcof{\righta(B),\ptcof{X,\righta(B \ptatimes X)}}
\]
        is once again a homeomorphism, since $X$ is compact.
        Taking $(\iota_{\righta(B)})^{-1}$ of the composition
        described above yields the expected pointed map
\[\righta(B) \Smash X \to \righta(B \ptatimes X)
\]
        and the formula given in the statement may be traced from the
        construction.

        For the second claim, the natural isomorphism
\[\ptcof{X \Smash Y,D} \ouriso \ptcof{X,\ptcof{Y,D}}
\]
        together with several variations of
        the natural bijection of Proposition
        \ref{almostadjoint} yields a string of natural bijections 
\[
\xymatrix{
{\Hom(B \ptatimes (X \Smash Y),D)} \ar[r] & 
    {\Hom(B,\ptcof{X \Smash Y,D})} \ar[d] \\
{\Hom(B \ptatimes X, \ptcof{Y,D})} \ar[r] &
{\Hom(B,\ptcof{X,\ptcof{Y,D}})} \\
{\Hom((B \ptatimes X) \ptatimes Y,D)} \ar[u].
}
\]
        Choosing $D=B \ptatimes (X \Smash Y)$ and taking the unique
        \starhom\ corresponding to the identity map of $D$
        yields the expected natural \starhom.
\end{proof}

\begin{lemma}
        \label{nullcyl}
        Suppose $A,B \in \ourcat$ with $A$ satisfying our second
        assumption. 
        Then the identity
        $\righta(B \ptatimes I) \to \righta(B \ptatimes I)$
        is null homotopic in $\pttop$.
\end{lemma}

\begin{proof}
        A pointed null homotopy for $I$,
\[H:I \Smash I_+ \to I
\]
        induces a morphism
\[B \ptatimes H: B \ptatimes (I \Smash I_+) \to B \ptatimes I
\]
        in $\ourcat$.  Now applying the functor $\righta$ and
        working with the two natural maps of Lemma \ref{htpy} yields
\[
\xymatrix{
\righta(B \ptatimes I) \Smash I_+ \ar[r] &
     \righta((B \ptatimes I)\ptatimes I_+) \ar[r] &
     \righta(B \ptatimes (I \Smash I_+))
     \ar[d]^{\righta(B \ptatimes H)} \\
&&     \righta(B \ptatimes I).
}
\]
        This composite is a pointed null
        homotopy of the identity on $\righta(B \ptatimes I)$.
\end{proof}

\begin{defn}\
        \label{retracts}
\begin{enumerate}

\item   A \starhom\  $i:B \to D$ is the inclusion of a retract if there
        exists a $\starhom$ $r: D \to B$ such that $r\circ i= id_B$.
        In the language of lifting diagrams, this is equivalent
        to the existence of a lift (the dotted arrow) in the following
        commutative square.
\[
\xymatrix{
{B} \ar[r]^{=} \ar[d]_{i} & {B} \ar[d] \\
{D} \ar[r] \ar@{.>}[ur] & {0}
}
\]

\item   An inclusion of a retract is the inclusion of a deformation
        retract if there exists a lift $h$ in the following diagram:
\[
\xymatrix{
{B} \ar[r]^{i} \ar[d]^{i} & {D}  \ar[r]^{j} & \ptcof{I_+,D} \ar[d]  \\
{D} \ar@{.>}[urr]^{h} \ar[r]_{(id,i\circ r)} & D \times D \ar[r]
  & \ptcof{S^0_+, D}
}
\]
        where precomposition by the map $I_+ \to S^0$
        which collapses the
        interval to the non-basepoint gives the \starhom\
\[j:D \approx \ptcof{S^0,D} \to \ptcof{I_+,D}.
\]

\end{enumerate}
\end{defn}

\begin{remark}
        This definition in terms of lifting diagrams, convenient
        for our application below,
        does agree with the usual notion of
        deformation retract for the category $\pttop$.
\end{remark}

\newcommand{\LLP}{left lifting property}

A morphism $f:A \to B$ will be said to have the \LLP\ with respect to a class of
morphisms if in each commutative square
\[
\xymatrix{
{A} \ar[r] \ar[d]_{f} & {D} \ar[d]^{p} \\
{B} \ar[r] \ar@{.>}[ur] & {E}
}
\]
with $p$ in the specified class of morphisms
there exists a dotted arrow (or lift) which makes both triangles
commute.  For example, for any topological space the map
\[X \to X \times I
\]
has the \LLP\ with respect to the class of Hurewicz fibrations.
Recall that the Serre cofibrations are defined as the maps in $\pttop$
with the \LLP\ with respect to the class of Serre fibrations that are also weak
homotopy equivalences. 
We will mean Serre fibration by the word fibration in what follows and
acyclic fibrations will refer to Serre fibrations which are also weak
homotopy equivalences.

\begin{lemma}
  \label{almostsm7}
  If $j:X \to Y$ is a Serre cofibration between compact spaces in
  $\pttop$ and $Z \in \pttop$, then 
\[j^*:\ptcof{Y,Z} \to \ptcof{X,Z}
\]
  is a Serre fibration in $\pttop$.
\end{lemma}

\begin{proof}
  For a convenient category, like Steenrod's category of compactly
  generated spaces, this would be the topological analog of Quillen's
  simplicial model category axiom SM7.  In that case, the result holds
  as in \cite{Hovey}.

  For Steenrod's category, we would instead have
\[k\ptcof{Y,k(Z)} \to k\ptcof{X,k(Z)}
\]
  is a fibration.  However, it follows from
  \cite[Lemma 5.3]{Steenrod(1967)} that $k(Z)$ may be replaced by $Z$ up to
  homeomorphism, since $Y$ and $X$ are compact.  Now notice the
  condition of being a fibration is checked by mapping out of
  compact spaces, so the $k$ on the outside is also
  not necessary. 
\end{proof}

The following two lemmas are the keys to the existence of the model
category structure.  The first is essentially due to Quillen.

\begin{lemma}
        \label{retractarg}
        Suppose $j:B \to D$ is a morphism in $\ourcat$ which has the \LLP\
        with respect to each morphism $p$ with $\righta(p)$
        a fibration in $\pttop$.  If $A$ satisfies our second assumption, 
        then $\righta(j)$ is a weak homotopy equivalence.
\end{lemma}

\begin{proof}
        We will show that $\righta(j)$ is the inclusion of a deformation
        retract, hence a weak homotopy equivalence in $\pttop$.
        To this end, it suffices to show that $j$ is the inclusion of
        a deformation retract, since by 
        our second assumption there
        is a natural homeomorphism
        $\righta(\ptcof{I_+,D}) \cong \ptcof{I_+,\righta(D)}$ and
        similarly for $\ptcof{S^0_+,D}$.

        The inclusion $S^0_+ \to I_+$ is a cofibration between compact
        spaces, hence
\[\ptcof{I_+,\righta(D)} \to \ptcof{S^0_+,\righta(D)}
\]
        is a fibration in $\pttop$ by Lemma \ref{almostsm7}.
        Thus, the homeomorphic (again by
        our second assumption) map
\[\righta(\ptcof{I_+,D}) \to \righta(\ptcof{S^0_+,D})
\]
        is a fibration as well.

        Now $j$ is assumed to have the \LLP\ with respect to
        \starhoms $p$ for which $\righta(p)$ is a fibration.
        We have just verified that $\ptcof{I_+,D} \to \ptcof{S^0_+,D}$
        is a \starhom\ of this type.  Since $\righta$ sends $0 \in \ourcat$ to
        the basepoint and every map $X \to *$ in $\pttop$ is a fibration,
        $B \to 0$ is also a \starhom\ of
        this type.  Hence, there exist lifts in the
        two diagrams which define $j$ as the inclusion of a deformation
        retract in $\ourcat$.
\end{proof}

\begin{lemma}
        \label{cellin}
        Suppose $A,B \in \ourcat$ and A satisfies our second
        assumption, while the canonical
        morphism $i_1: B \to B \ptatimes I$ has the
        \SEP[]. Furthermore, assume $j:B \to D$ is a morphism in
        $\ourcat$ which has the \LLP\ with respect to any morphism $q$
        with $\righta(q)$ an
        acyclic fibration in $\pttop$. Then $j$ has the \SEP[].
\end{lemma}

\begin{proof}
        By Lemma \ref{nullcyl}, one has that
        $\righta$ applied to the \starhom\
        $B \ptatimes I \to 0$ is an acyclic fibration in $\pttop$.
        Now, by assumption there exists a lift in the diagram
\[\xymatrix{
B \ar[r] \ar[d]^{j}& {B\ptatimes I} \ar[d]\\
D \ar[r] \ar@{.>}[ur] & 0
}\]
        which implies $j$ is the first factor in a factorization of
        $i_1:B \to B \ptatimes I$, hence it has the \SEP by Lemma
        \ref{firstfactor}.
\end{proof}

We say the functor $\righta$ creates weak equivalences if $p$ is a
weak equivalence in $\ourcat$ precisely when $\righta(p)$ is a weak
equivalence in $\pttop$ and similarly for fibrations.
We also remind the reader that all three of the assumptions here are
satisfied by the object
$A=qA' \atensorK \compacts$ for any separable $A' \in \calg$, as
discussed at the beginning of this section.

\begin{prop}
        \label{liftone}
        Suppose $A \in \ourcat$ satisfies:
\begin{enumerate}
\item   $\righta=\Hom(A,?)$ has $\lefta=A \ptatimes ?$ as a partial left adjoint
\item   there is a natural homeomorphism
  \[\righta(\ptcof{X,B}) \to \ptcof{X,\righta(B)}
  \]
\item   $A$ has the \sCSEP[].
\end{enumerate}
        Then
        there is a cofibrantly generated model structure on
        $\ourcat$ with fibrations and weak equivalences created by
        $\righta$ and $A$ itself cofibrant. 
        Furthermore, this structure
        is right proper (all objects are fibrant), and the
        generating cells are of the form $\lefta$ applied to the
        inclusions $i':S^{n-1}_+ \to D^n_+$.
\end{prop}

This proof has become standard, with the exception of the fact that
$(\lefta,\righta)$ is not quite an adjoint pair.  Thus, we present the
proof in full detail.  Compare with the proof of \cite[8.7]{Michele}
and notation from \cite{Hovey}.

\begin{proof}
        The axioms (from \cite[Definition 3.3]{Quillenrational}) will
        be verified directly.
\begin{enumerate}

\item   The existence of limits and colimits is handled by
        Lemmas \ref{limits} and \ref{colimits}.

\item   The 2 of 3 property for weak equivalences follows from that
        of weak equivalences in $\pttop$.

\item   The class of cofibrations is defined as the retracts of relative
        $I$-cell complexes, hence is closed under retracts by
        definition.
        Here $I$ is the set of generating cells, $\lefta$ applied to the
        inclusions $i':S^{n-1}_+ \to D^n_+$.
        The fact that $\righta$ preserves retracts (as a functor)
        and the fact that the classes of fibrations and
        weak equivalences in $\pttop$ are closed under
        retracts implies the same for $\ourcat$.

\item   Let 
\[
\xymatrix{
{B} \ar[d]_{i} \ar[r] & {X} \ar[d]_{p} \\
{D} \ar[r] & {Y}
}
\]
        be a potential lifting square.

        If $i$ is a retract of a relative $I$-cell complex, and $p$ is
        an acyclic fibration, one must verify the existence of a lift.
        Consider first the case that $i$ is a generating cofibration.
        Then $i=\lefta(i')$ for a generating cofibration $i'$ in
        $\pttop$ (whose source and target lie in $\ptcpt$).
        Thus, $i'$ has the \LLP\ with respect to
        $\righta(p)$, which is an acyclic fibration in $\pttop$
        by assumption.  Considering the unique square
\[
\xymatrix{
{S^{n-1}_+} \ar[d]_{i'} \ar[r] & {\righta(X)} \ar[d]^{\righta(p)} \\
{D^{n}_+} \ar[r] \ar@{.>}[ur]^{f} & {\righta(Y)}
}
\]
        weakly adjoint to the original square, there must exist a lift $f$,
        which implies there was a lift in the original square via
        the weak adjoint of $f$.  In building a retract of a relative $I$-cell
        complex, one uses cobase change (pushout), coproducts,
        transfinite composition (sequential
        colimits) and retracts, all operations which preserve the
        \LLP\ with respect
        to a class of morphisms.  Hence, the arbitrary cofibration
        case follows from that of the generating cofibration.

        Now suppose that $i$ is both a retract of a relative $I$-cell
        complex and a weak equivalence, while $p$ is a fibration.
        It is shown below in (5) that $i$ can be
        factored as a relative $J$-cell complex $j:B \to F$
        followed by a fibration $q:F \to D$, where $J$ is the set of
        maps $\lefta(j')$ for $j':D^n_+ \to (D^n \times I)_+$ the inclusion
        of the base of the cylinder with a disjoint basepoint added.
        The argument in the previous paragraph then
        implies that $j$ has the \LLP\ with respect to all fibrations
        by construction, hence is a weak equivalence by Lemma \ref{retractarg}.
        This \LLP\ with respect to all fibrations also implies the
        existence of a morphism $h:F \to X$ which makes the right
        portion of
\[
\xymatrix{
{B} \ar[dr]^{j} \ar[dd]_{i} \ar[rr] & {} & {X} \ar[dd]^{p} \\
  & {F} \ar[d]^{q} \ar@{.>}[ur]^{h}   \\
{D} \ar[r]^{=} \ar@{.>}[ur]^{f} & {D} \ar[r] & {Y}
}
\]
        commute.  Also, the
        2 of 3 property established in (2) implies that $q$ is actually
        an acyclic fibration.  However, the argument in the previous
        paragraph then
        implies there is a lift of $i$ against $q$, $f:D \to F$ making
        the left portion of the previous diagram commute.  Now
        the composition $hf$ acts as a lift in the original square.

\item   Suppose $k:X \to Y$ is a morphism in $\ourcat$.  Construct a
        relative $I$-cell complex $X=X_0 \to X_\gamma$ as follows.
        (Here $\gamma$ is a cardinal such that each
        of the $\lefta(S^{n-1}_+)$ is $\gamma$-small
        with respect to morphisms in $\ourcat$ satisfying the \SEP[],
        which is possible by Lemma \ref{allsmall_in_ourcat}.)
        For limit ordinals $\delta < \gamma$, define
        $X_\delta=\underset{\sigma < \delta}{\colim} X_\sigma$.  For
        successor ordinals, define $X_{\delta+1}$ to be the pushout in
        the diagram
\[
\xymatrix{
{\coprod \lefta(S^{n-1}_+) } \ar[d] \ar[r]^{}
        & {\coprod \lefta(D^{n}_+)} \ar[d] \\
{X_\delta} \ar[r] & {X_{\delta+1}}
}
\]
        where the coproducts are indexed over the set of commutative
        squares of the form
\[
\xymatrix{
{\lefta(S^{n-1}_+)} \ar[d] \ar[r] & {\lefta(D^{n}_+)} \ar[d] \\
{X_\delta} \ar[r] & {Y}.
}
\]
        The universal properties of colimits and pushouts then yields
        a factorization of $k$ as $X \to X_\gamma \to Y$ with
        $X \to X_\gamma$ a relative $I$-cell complex.  It remains to
        verify that the morphism $p:X_\gamma \to Y$ is an
        acyclic fibration.  Since $\pttop$ is a cofibrantly generated
        model category,
        it suffices to verify that $\righta(p)$ has the RLP with respect
        to each $S^{n-1}_+ \to D^{n}_+$.  Equivalently, one can consider
        the weak adjoint of each such diagram and show that each
\[\lefta(i'):\lefta(S^{n-1}_+ \to \lefta(D^{n}_+)
\]
        has the \LLP\ with respect to $p$.
        Thus, we consider a lifting diagram
\[
\xymatrix{
{\lefta(S^{n-1}_+)} \ar[d]_{i} \ar[r]^{f} & {X_\gamma} \ar[d]_{p} \\
{\lefta(D^{n}_+)} \ar[r] & {Y}.
}
\]
        The object $\lefta(S^{n-1}_+)$ is
        $\gamma$-small with respect to morphisms in $\ourcat$ satisfying
        the \SEP by assumption, while every
        relative $I$-cell complex satisfies the \SEP by
        the combination of Lemmas \ref{embedcolim},
        \ref{SEPcobase} and \ref{cellin} with Corollary
        \ref{SEPcoprod}.  Thus,
        one may factor $f$ as
\[
\xymatrix{
{\lefta(S^{n-1}_+)} \ar[r]^{g} & X_\delta \ar[r]^{h} & X_\gamma
}
\]
        for some $X_\delta$ with $\delta < \gamma$.
        However, one then has a commutative diagram of the form
\[
\xymatrix{
{\lefta(S^{n-1}_+)} \ar[r]^{g} \ar[d]_{i} & {X_\delta} \ar[d]^{ph} \\
{\lefta(D^{n}_+)} \ar[r] & {Y}
}
\]
        which must have been used to build $X_{\delta + 1}$.  Hence,
        the composition
\[\lefta(D^{n}_+) \to X_{\delta + 1} \to X_\gamma
\]
        acts as a lift in the original square
\[
\xymatrix{
{\lefta(S^{n-1}_+)} \ar[r]^{g} \ar[d]_{i} & {X_\delta} \ar[d] \ar[r]^{h}
     & {X_\gamma} \ar[d]^{p}\\
{\lefta(S^{n-1}_+)} \ar[r] & {X_{\delta+1}} \ar[ur] \ar[r] & {Y}.
}
\]

        The other factorization of $k$ is produced similarly, using $J$ in
        place of $I$, with one additional difficulty.  Namely,
        one must verify that the relative $J$-cell complex $j$ produced
        is actually a retract of a relative $I$-cell complex as well as
        a weak equivalence.  The lifting argument from the proof of
        the first half of (4) implies that $j$ has the \LLP\ with
        respect to each fibration, hence is a weak equivalence by
        Lemma \ref{retractarg}.
        Now factor $j$ as a relative $I$-cell $i$ followed by an acyclic
        fibration, as in the previous paragraph. 
        Then the \LLP\ just mentioned for $j$
        implies the existence of a lift in the square
\[
\xymatrix{
{B} \ar[r]^{i} \ar[d]_{j} & {F} \ar[d]^{q} \\
{D} \ar[r]_{=} \ar@{.>}[ur]  & {D}
}
\]
        which exhibits $j$ as a retract of $i$, hence as a retract of
        a relative $I$-cell complex.
\end{enumerate}
\end{proof}

There is a relatively straightforward notion of intersection of cofibrantly
generated model category structures introduced in \cite{Michele}. 
This makes the following lifting lemma a corollary of the previous proposition.
We say the collection of functors $\{ \righta \}$ creates the
fibrations (resp. weak equivalences) when a \starhom\ $p$ is a
fibration if and only if each
$\righta(p)$ is a fibration (resp. weak equivalence) in $\pttop$.

\begin{cor}
        \label{liftall}
        Suppose $\{A_\alpha\}$ is a set of objects in $\ourcat$ all
        satisfying the three conditions in Proposition \ref{liftone}.
        Then there
        is a cofibrantly generated model structure on $\ourcat$ with
        fibrations and weak equivalences created by the set of
        functors $\{ \righta \}$ and with the property that each
        $A_\alpha$ is cofibrant.
\end{cor}

\begin{proof}
        This follows from the intersection construction of
        \cite[Proposition 8.7]{Michele}.
        The relative smallness follows from the fact that cell complexes
        built from the collection of generating cells
        in each structure have the \SEP and Lemma
        \ref{allsmall_in_ourcat}. The other condition is handled by Lemma
        \ref{retractarg} and the uniformity of the cylinder chosen
        for each structure.
\end{proof}

The following technical lemmas will prove valuable in identifying weak
equivalences in specific examples considered in the next section.

\begin{lemma}
        \label{modelcyl}
        Given any $B \in \ourcat$, one has $\ptcof{I_+,B}$ acting as a path
        object in the model category sense for any model structure
        constructed with Proposition \ref{liftone} or Corollary
        \ref{liftall}.
\end{lemma}

\begin{proof}
        Consider the composition $i:S^0_+ \to I_+$ followed by
        $f:I_+ \to S^0$
        in $\pttop$, which gives a factorization of the fold map as
        a cofibration between compact spaces followed by a homotopy
        equivalence.
        Then precomposition with $i$,
\[i^*:\ptcof{I_+,\righta(B)} \to \ptcof{S^0_+,\righta(B)}
\]
        yields a fibration by Lemma \ref{almostsm7} and precomposition with $f$,
\[f^*:\ptcof{S^0,\righta(B)} \to \ptcof{I_+,\righta(B)}
\]
        yields a homotopy equivalence, hence a weak equivalence. 
        However, our second assumption on $A$
        then implies $\ptcof{I_+,B} \to \ptcof{S^0_+,B}$ is a fibration
        in $\ourcat$ and $\ptcof{S^0,B} \to \ptcof{I_+,B}$ is a weak
        equivalence in $\ourcat$.  Since
        $\ptcof{S^0,B} \approx B$, and
        $\ptcof{S^0_+,B} \approx B \times B$,
        this yields a factorization of the diagonal map
\[B \to \ptcof{I_+,B} \to B \times B
\]
        with the first morphism
        a weak equivalence and the second map a fibration in
        $\ourcat$.  This is
        the definition of a path object in a model category.
\end{proof}

\begin{lemma}
        \label{pizero}
        Suppose $A_\alpha \in \ourcat$ is one of the generating
        objects in a structure given by Proposition \ref{liftone} or
        Corollary \ref{liftall}.  Then
        $\pi_0 \Hom(A_\alpha,B)$,
        $[A_\alpha,B]$ and
        $\ho(A_\alpha,B)$  are all
        naturally isomorphic as sets. 
\end{lemma}

\begin{proof}
        In order to see that
$\pi_0 \Hom(A_\alpha,B) \approx
        [A_\alpha,B]$  it suffices
        to note that there is, by our second assumption on $A_\alpha$
        a natural homeomorphism
        $\ptcof{I_+,\Hom(A_\alpha,B)} \approx
        \Hom(A_\alpha,\ptcof{I_+,B})$
        and similarly for $S^0$ in place of $I_+$.  Thus, one has the
        same set modulo the same equivalence relation on both sides,
        since $\ptcof{I_+,B} = \cof[I,B]$.

        In order to see that $[A_\alpha,B] \approx
        \ho(A_\alpha,B)$ under these
        conditions, it suffices to note that
        $\cof[I,B] = \ptcof{I_+,B}$ acts as a path object in the
        model category sense by Lemma \ref{modelcyl}.
\end{proof}

\section{The Main Theorems}
\label{both}

We now have the technical tools to build our required model category
structure.

\begin{defn}[$\K$-model category structure]
We define the \K-theory model category structure on $\ourcat$ to be
the model category structure obtained from
Proposition \ref{liftone} using the object $q\C \atensorK \compacts$.
A weak equivalence for the $\K$-model category structure
we will refer to as a weak $\K$-equivalence.
\end{defn}

\begin{thm}
  \label{detectK}
 A morphism $f: B \to B'$ of \calgs\ is a weak \K-equivalence if and only if
it is a $K_*$-equivalence, i.e.~if the induced map $K_*(B) \to K_*(B')$ is
an isomorphism. Moreover, for \calgs\ $B$ there is a natural isomorphism
\[ Ho(\C, B) \cong K_0(B).\]
\end{thm}

\begin{proof}
  By Proposition \ref{liftone},
  $f$ is a weak $\K$-equivalence if and only if the
  induced map
\[
  Hom(q\C \atensorK \compacts, B) \to Hom(q\C \atensorK \compacts, B')
\]
  is a weak homotopy equivalence.
 
  For any $\sigma$-unital
  \calgs[ ]$B$ we have the following natural sequence of isomorphisms
\begin{eqnarray*}
\lefteqn{
  \pi_n(Hom(q\C, B \tensor\compacts)) \cong \pi_0 Hom(q\C, \ptcof{S^n,B \tensor \compacts})
  \cong}\\
  & &\pi_0 Hom(q\C, \ptcof{S^n,B} \tensor \compacts) \cong 
  KK_0(\C,\ptcof{S^n,B}) \cong K_0(\ptcof{S^n,B})\cong K_n(B).
\end{eqnarray*}
  The first isomorphism results from Proposition \ref{cadjoint},
  the second isomorphism is induced by the isomorphism
  $\ptcof{S^n,B}\tensor\compacts \cong \ptcof{S^n, B \tensor\compacts}$
  of Lemma \ref{take_tensor_K_out_of_cof},
  the third one is Cuntz's Theorem \ref{htpyKK}, and the last two are standard.
 
  Hence using Proposition \ref{reduce_to_pi-iso} we obtain 
  that a map of $\sigma$-unital \calgs\ is a weak $\K$-equivalence if and only if
  it is a $K_*$-equivalence. Thus, we have verified the first statement in case 
  $B$ are $B'$ are $\sigma$-unital \calgs.

  For an arbitrary \calgs[] $B'$, consider the directed
  system of $\sigma$-unital sub-\calgs\ $B$ of $B'$.
  Since $q\C$ is separable the image
  of a \starhom\ into $B'\tensor \compacts$ is contained in
  $B\tensor \compacts$ for a separable (and hence $\sigma$-unital)
  sub-\calgs[ ] $B$ of $B'$, so
  \[\colim_B Hom(q\C, B \tensor \compacts) \cong Hom(q\C,B'\tensor \compacts),\]
  where the colimit is taken over the system of
  $\sigma$-unital sub-\calgs\ $B$ in $B'$.
  Continuity of the functor $K_*$
  then yields
  \[ \pi_*Hom(q\C,B'\tensor \compacts) \cong \colim_B \pi_*Hom(q\C, B\tensor \compacts)
    \cong  \colim_B K_*(B) \cong K_*(B')\]
  and the general case follows. Here we interchanged taking homotopy groups
  with taking a colimit. In general the two operations do not commute, but
  in our special case they do. The reason being that
  for a separable compact Hausdorff space $X$ the canonical map
\[
  \colim_B \cof[X, Hom(q\C, B \otimes \compacts)] \to \cof[X, Hom(q\C, B' \tensor \compacts)]
\]
  is a continuous bijection.
  Injectivity of this map is clear. To see surjectivity let
  $\underline{X}=\{x_1, x_2, ..\}$ be a countable dense set of $X$,
  $\underline{q\C}=\{q_1,q_2,...\}$ be a countable dense set of $q\C$, and let
  the functions $x^{B'}_{ij}: B' \tensor \compacts \to B'$ and
  the compact operators $e_{ij}\in \compacts$ be defined as in the proof of Lemma
  \ref{take_tensor_K_out_of_cof}. Given a map
  $f:X \to  Hom(q\C, B' \tensor \compacts)$ define
  the subset $\underline{B}_f \subset  B'$ by
\[ \underline{B}_f= \{ x^{B'}_{ij}(f(x_k)(q_l)) \ |\ i,j,k,l\in \N\}
\]
  Then $\underline{B}_f$ is countable and therefore generates
  a separable sub-\cstar-algebra
  $B_f \subset B'$, and the given map $f$ is in the image of
  $\cof[X, Hom(q\C, B_f \otimes \compacts)]$.

  To prove the second statement of the theorem, note that
 \[
  Hom(q\C \atensorK \compacts,q\C \atensorK \compacts) \to
  Hom(q\C \atensorK \compacts,q\C)
\]
  given by composing with the
  map $q_{q\C}: q\C \atensorK \compacts \to q\C$ 
  (defined in \ref{introduce_p_C})
  is a weak equivalence by Proposition \ref{verify_cofibrant_replacement}.
  Since $q\C \atensorK \compacts$ is cofibrant by Proposition
  \ref{liftone}, it then forms a cofibrant replacement of $q\C$.
  Moreover, the canonical map
  $q\C \to \C$ is a $K_*$-equivalence. So, using the first statement
  above, it follows that the composition
  $q\C \atensorK \compacts \to q\C \to \C$ is a cofibrant replacement for $\C$.
  Consequently, Lemma \ref{pizero} followed by Proposition \ref{tensorK_adjoint_is_weq} and
  Cuntz's Theorem \ref{htpyKK} imply
  \[ Ho(\C,B) \cong \pi_0(Hom(q\C \atensorK \compacts, B))
  \cong \pi_0(Hom(q\C, B \tensor \compacts))\cong K_0(B).\]
\end{proof}

\begin{remark}
The functor $B \mapsto K_0(B)$ does not behave well with respect to
inverse limits, which is the reason why we restricted our attention to
\calgs\ in the second statement of Theorem \ref{detectK}.
\end{remark}

We now change to a different model structure on $\ourcat$,
which we will refer to as the \KK-structure. 

\begin{defn}[$\KK$-model category structure]\label{KK-model cat}
We define the $\KK$-model category structure on $\ourcat$ to
be the model category structure which comes from Corollary
\ref{liftall} using the \lmccstar-algebras $qA \atensorK \compacts$
as $A$ varies over a set of representatives, one for each
isomorphism class of separable $\calgs$. 
Clearly, there is only a set
of isomorphism classes of separable $\calgs$ since the separable
condition is equivalent to having a representation in a
countable-dimensional Hilbert space.
A weak equivalence for  the $\KK$-model category structure
we call a weak $\KK$-equivalence.
\end{defn}

\begin{thm} \label{generalKK}
  A \starhom\ $f: B \to B'$ of $\sigma$-unital \calgs\ is a weak $\KK$-equivalence
if and only if it is a ``(left) $KK_*$-equivalence'', i.e.~if the induced map $KK(A,B) \to KK(A,B')$
is an isomorphism for all separable \calgs\ $A$. 
Moreover, for separable \calgs\ $A$ and
$\sigma$-unital \calgs\ $B$ there is a natural
  isomorphism
\[
  Ho(A,B) \cong KK(A,B).
\]
\end{thm}

\begin{proof}
  By definition
  a \starhom\ $f: B \to B'$ is a weak $\KK$-equivalence if and only if the induced maps
\[
  Hom(qA \atensorK \compacts, B) \to Hom(qA \atensorK \compacts, B')
\]
  are weak equivalences for all separable \calgs\ $A$.
By \ref{reduce_to_pi-iso} this is the case if and only if 
the maps $\pi_n Hom(qA ,B \tensor \compacts) \to \pi_n Hom(qA , B'\tensor \compacts )$
are isomorphisms for all separable $A\in\calg$ and all $n\in\N$.
However, by Cuntz's Theorem \ref{htpyKK} we have the 
following isomorphism
\[
  \pi_*Hom(qA, B \tensor \compacts) \cong  KK_*(A,B),
\]
natural for $\sigma$-unital \calgs\ $B$.  
Thus, $f$ is a weak $\KK$-equivalence if and only if it is a (left)
$KK_*$-equivalence.

To see the second statement, recall from Proposition \ref{verify_cofibrant_replacement}
that composition with $q_{qB}: qB \atensorK \compacts \to qB$
(defined in \ref{introduce_p_C}) induces weak equivalences
\[
  Hom(qA \atensorK \compacts, qB \atensorK \compacts) \to Hom(qA \atensorK \compacts, qB)
\]
for all separable \calgs\ $A$. Moreover, it is well-known that
$qA \to A$ is a (left) $KK_*$-equivalence for all separable \calgs\ $A$ (cf.~\cite[19.1.2.(i)]{Blackadar(1998)}).
Thus, the composition $qA \atensorK \compacts \to qA \to A$ is a cofibrant replacement
for $A$ if $A$ is a separable \calgs[.] Consequently, as in the proof of Theorem \ref{detectK} we get
\[ Ho(A,B) \cong \pi_0 Hom(qA \atensorK \compacts, B) \cong \pi_0 Hom(qA, B \tensor \compacts) \cong KK(A,B).\]
\end{proof}

\begin{remark}\label{extendingKK}
When defining $KK(A,B)$ one typically assumes $A$ to be a separable \calgs[] and
$B$ to be a $\sigma$-unital \calgs[]. $KK$-groups also have been defined in more general
settings by various authors. A new approach for defining it for algebras $A,B\in\ourcat$ 
is given by simply defining it by $\KK(A,B)=Ho(A,B)$.
\end{remark}

    %
      \bibliographystyle{amsalpha}
      \bibliography{KKcites}

\providecommand{\bysame}{\leavevmode\hbox to3em{\hrulefill}\thinspace}
\providecommand{\MR}{\relax\ifhmode\unskip\space\fi MR }
\providecommand{\MRhref}[2]{%
  \href{http://www.ams.org/mathscinet-getitem?mr=#1}{#2}
}
\providecommand{\href}[2]{#2}
\begin{thebibliography}{tDKP70}

\bibitem[Are52]{Arens(1952)}
Richard Arens, \emph{A generalization of normed rings}, Pacific J. Math.
  \textbf{2} (1952), 455--471.

\bibitem[Are58]{Arens(1958)}
\bysame, \emph{Dense inverse limit rings}, Michigan Math. J \textbf{5} (1958),
  169--182.

\bibitem[Bla98]{Blackadar(1998)}
Bruce Blackadar, \emph{{$K$}-theory for operator algebras}, Mathematical
  Sciences Research Institute Publications, vol.~5, Cambridge University Press,
  Cambridge, 1998.

\bibitem[Bor94]{Borceux(1994a)}
Francis Borceux, \emph{Handbook of categorical algebra. 1}, Encyclopedia of
  Mathematics and its Applications, vol.~50, Cambridge University Press,
  Cambridge, 1994.

\bibitem[Cun87]{Cuntz(1987)}
Joachim Cuntz, \emph{A new look at {$KK$}-theory}, $K$-Theory \textbf{1}
  (1987), no.~1, 31--51.

\bibitem[Dix77]{Dixmier(1977)}
Jacques Dixmier, \emph{{$C\sp*$}-algebras}, North-Holland Publishing Co.,
  Amsterdam, 1977, Translated from the French by Francis Jellett, North-Holland
  Mathematical Library, Vol. 15.

\bibitem[Hir03]{Hirschhorn(2003)}
Philip~S. Hirschhorn, \emph{Model categories and their localizations},
  Mathematical Surveys and Monographs, vol.~99, American Mathematical Society,
  Providence, RI, 2003.

\bibitem[Hov99]{Hovey}
Mark Hovey, \emph{Model categories}, Mathematical Surveys and Monographs,
  vol.~63, American Mathematical Society, Providence, RI, 1999.

\bibitem[IJ02]{Michele}
Michele Intermont and Mark~W. Johnson, \emph{Model structures on the category
  of ex-spaces}, Topology Appl. \textbf{119} (2002), no.~3, 325--353.

\bibitem[KN63]{Kelley-Namioka(1963)}
J.~L. Kelley and Isaac Namioka, \emph{Linear topological spaces}, With the
  collaboration of W. F. Donoghue, Jr., Kenneth R. Lucas, B. J. Pettis, Ebbe
  Thue Poulsen, G. Baley Price, Wendy Robertson, W. R. Scott, Kennan T. Smith.
  The University Series in Higher Mathematics, D. Van Nostrand Co., Inc.,
  Princeton, N.J., 1963.

\bibitem[Mal86]{Mallios(1986)}
Anastasios Mallios, \emph{Topological algebras. {S}elected topics},
  North-Holland Mathematics Studies, vol. 124, North-Holland Publishing Co.,
  Amsterdam, 1986, Notas de Matem\'atica [Mathematical Notes], 109.

\bibitem[Mey00]{Meyer(2000)}
Ralf Meyer, \emph{Equivariant {K}asparov theory and generalized homomorphisms},
  $K$-Theory \textbf{21} (2000), no.~3, 201--228, archiv:math.KT/0001094.

\bibitem[Mic52]{Michael(1952)}
Ernest~A. Michael, \emph{Locally multiplicatively-convex topological algebras},
  Mem. Amer. Math. Soc., \textbf{1952} (1952), no.~11, 79.

\bibitem[ML98]{MacLane}
Saunders Mac~Lane, \emph{Categories for the working mathematician}, second ed.,
  Graduate Texts in Mathematics, vol.~5, Springer-Verlag, New York, 1998.

\bibitem[Phi88a]{Phillips(1988a)}
N.~Christopher Phillips, \emph{Inverse limits of {$C\sp *$}-algebras}, J.
  Operator Theory \textbf{19} (1988), no.~1, 159--195.

\bibitem[Phi88b]{Phillips(1988b)}
\bysame, \emph{Inverse limits of {$C\sp *$}-algebras and applications},
  Operator algebras and applications, Vol.\ 1, London Math. Soc. Lecture Note
  Ser., vol. 135, Cambridge Univ. Press, Cambridge, 1988, pp.~127--185.

\bibitem[Phi89]{Phillips(1989)}
\bysame, \emph{Classifying algebras for the {$K$}-theory of {$\sigma$}-{$C\sp
  *$}-algebras}, Canad. J. Math. \textbf{41} (1989), no.~6, 1021--1089.

\bibitem[Qui69]{Quillenrational}
Daniel Quillen, \emph{Rational homotopy theory}, Ann. of Math. (2) \textbf{90}
  (1969), 205--295.

\bibitem[Ste67]{Steenrod(1967)}
N.~E. Steenrod, \emph{A convenient category of topological spaces}, Michigan
  Math. J. \textbf{14} (1967), 133--152.

\bibitem[tDKP70]{tomDieck-Kamps-Puppe(1970)}
Tammo tom Dieck, Klaus~Heiner Kamps, and Dieter Puppe, \emph{Homotopietheorie},
  Springer-Verlag, Berlin, 1970, Lecture Notes in Mathematics, Vol. 157.

\end{thebibliography}

\end{document}